\documentclass{jloganal}
\usepackage{amsmath,amsthm,amssymb,amscd}
\title[Interpreter for topologists]{Interpreter for topologists}

\author[Jind{\v r}ich Zapletal]{Jind{\v r}ich Zapletal}
\givenname{Jind{\v r}ich}
\surname{Zapletal}
\address{Department of Mathematics\\
University of Florida\\\newline
Gainesville FL 32611}
\email{zapletal@math.ufl.edu}
\keywords{forcing, interpretation, topology}
\subject{primary}{msc2010}{03E40}
\subject{secondary}{msc2010}{54B99}
\volumenumber{}
\issuenumber{}
\publicationyear{}
\papernumber{}
\startpage{}
\endpage{}
\doi{}
\MR{}
\Zbl{}
\received{}
\revised{}
\accepted{}
\published{}
\publishedonline{}
\proposed{}
\seconded{}
\corresponding{}
\editor{}
\version{}

\newcommand{\ga}{\alpha}
\newcommand{\gb}{\beta}

\newcommand{\gd}{\delta}
\newcommand{\gw}{\omega}

\newcommand{\gS}{\Sigma}
\newcommand{\gs}{\sigma}
\newcommand{\eps}{\varepsilon}

\newcommand{\liff}{\leftrightarrow}

\newcommand{\baire}{\gw^\gw}
\newcommand{\bintree}{2^{<\gw}}
\newcommand{\gwtree}{\gw^{<\gw}}

\newcommand{\supp}{\mathrm{supp}}

\newcommand{\dom}{\mathrm{dom}}

\newcommand{\rng}{\mathrm{rng}}
\newcommand{\power}{\mathcal{P}}

\newtheorem{theorem}{Theorem}[section]

\newtheorem{claim}[theorem]{Claim}
\newtheorem{corollary}[theorem]{Corollary}
\newtheorem{fact}[theorem]{Fact}
\newtheorem{proposition}[theorem]{Proposition}

\theoremstyle{definition}
\newtheorem{definition}[theorem]{Definition}
\newtheorem{example}[theorem]{Example}
\newtheorem{question}[theorem]{Question}

\begin{document}

\begin{abstract}
Let $M$ be a transitive model of set theory. There is a canonical interpretation functor between the category of regular Hausdorff, continuous open images of {\v C}ech-complete spaces of $M$ and the same category in $V$, preserving many concepts of topology, functional analysis, and dynamics. The functor can be further canonically extended to the category of Borel subspaces. This greatly simplifies and extends similar results of Fremlin.
\end{abstract} 

\maketitle

\section{Introduction}
A powerful trick set theorists may employ in proving a statement $\phi$ with parameters is to move to a different model of set theory, prove that $\phi$ holds there, and then pull back the statement to the original universe. The trick requires knowledge about how
the parameters and the formula $\phi$ survive the transportation between various models of set theory. While there are deep results in this direction such as Shoenfield absoluteness \cite[Theorem 25.20]{jech:set} or Woodin's $\Sigma^2_1$ absoluteness \cite[Theorem 3.2.1]{larson:book}, the current  wave of applications
of logic and set theory to complicated topological structures seems to present new challenges here. In this paper, I will show that for transitive models $M\subset V$ of set theory with the axiom of choice, there is an intepretation functor from a broad category of topological
structures in the model $M$ to a similar category in $V$. The functor satisfies most if not all reasonable demands on canonicity, it commutes with most natural topological operations, and happily interacts with many fundamental theorems
of various areas of mathematics. As a result, a rich theory is obtained that supports and clarifies various absoluteness tricks popular among set theorists, and extends them to a much larger category of spaces than the usual Polish spaces.

The definition of interpretation of a topological space is natural, but necessarily a little verbose. It is reminiscent of the construction of {\v C}ech-Stone compactification. Just to be explicit, a topological
space is a pair $\langle X,{\tau}\rangle$ where $X$ is a set and ${\tau}$ is a collection of its ``open" subsets which is closed under arbitrary unions and finite intersections. I will always require the topological spaces in this paper to be $\mathrm{T}_0$.

\begin{definition}
Suppose that $M$ is a transitive model of set theory, $M\models \langle X, {\tau}\rangle$ is a topological space. A \emph{topological preinterpretation} of $X$ (over $M$) is a topological space $\langle\hat X, \hat {\tau}\rangle$ together with a map $\pi\colon X\to\hat X$ and $\pi\colon {\tau}\to\hat {\tau}$ such that

\begin{enumerate}
\item for every point $x\in X$ and every set $O\in {\tau}$, $x\in O$ if and only if $\pi(x)\in\pi(O)$;
\item $\pi$ commutes with finite intersections and arbitrary unions of open sets in the model $M$: if $M\models O=\bigcup_{i\in I}O_i$ where
$O, O_i\in {\tau}$, then $\pi(O)=\bigcup_{i\in I}\pi(O_i)$;
\item $\pi(0)=0$ and $\pi(X)=\hat X$;
\item $\pi''{\tau}$ is a basis of the topology $\hat {\tau}$.
\end{enumerate}
\end{definition}

\noindent It is easy to observe that the two parts of the map $\pi$ can be reconstructed from each other if (1-3) hold, which justifies using the same letter for the map on points and on open sets in the model $M$.
There are many preinterpretations of a given topological space and it is necessary to organize them.

\begin{definition}
Suppose that $M\models \langle X, {\tau}\rangle$ is a topological space. Suppose that $\pi_0\colon X\to\hat X_0$ and $\pi_1\colon X \to\hat X_1$
are two preinterpretations of $X$. Say that $\pi_0\leq \pi_1$ if there is a \emph{reduction} $h\colon\hat X_0\to\hat X_1$; this is a map such that $\pi_1=h\circ\pi_0$ and $h^{-1}\pi_1(O)=\pi_0(O)$ for every open set $O\in {\tau}$. Say that $\pi_0, \pi_1$ are \emph{equivalent} if there is a reduction $h$ which is at the same time a bijection of $\hat X_0$ and $\hat X_1$.
\end{definition}

It is not difficult to show that for preinterpretations $\pi_0, \pi_1$, they are equivalent if and only if each of them is reducible to the other. I will not distinguish between two equivalent preinterpretations. The stage is now set for the definition of an interpretation as the most complicated  preinterpretation
of a given space.

\begin{definition}
\label{i1definition}
Suppose that $M\models \langle X, {\tau}\rangle$ is a topological space. An \emph{interpretation} of $X$ is the $\leq$-largest preinterpretation, if it exists.
\end{definition}

\begin{definition}
\label{i2definition}
Suppose that $M\models X,Y$ are topological spaces and $f\colon X\to Y$ is a continuous function. Suppose that $\pi\colon X\to\hat X$ and $\chi\colon Y\to\hat Y$ are interpretations. An interpretation $\hat f$
of $f$ is a continuous function from $\hat X$ to $\hat Y$ such that $\hat f(\pi(x))=\chi(y)$ whenever $f(x)=y$.
\end{definition}

The obvious question regarding the existence and uniqueness of interpretations is answered in the affirmative in the very wide case of regular Hausdorff spaces. However, one also has to consider the general expectation that the notion of interpretation will commute with the most usual topological operations. To satisfy expectations of this kind, it appears to be necessary to restrict attention to the category of interpretable spaces: the regular Hausdorff continuous open images of {\v C}ech-complete spaces.
In this category, I develop a completely harmonious theory of interpretations, starting with continuous functions, interpretations of Borel sets, subspaces, product etc.\ all the way to interpretations of structures such as duals of Banach spaces etc. I supply a long list of natural operations which commute with the interpretation functor, and a long list of properties which are inherited by the interpretations from their original spaces and structures.

To deal with the very large class of spaces which are not interpretable, I develop the notion of an interpretable Borel space: this is a topological space with a Borel structure which is a Borel subspace of an interpretable space. An interpretation of a Borel interpretable space is required to commute not only with finite intersections and arbitrary unions of open sets, but also with complements and countable unions and intersections of Borel sets. Many expected commutativity properties do hold for the interpretation functor on the class of Borel interpretable spaces.

In preexisting work, Fremlin \cite{fremlin:top} described a similar interpretation functor for topological spaces in the special case where $V$ is a generic extension of $M$. The Fremlin interpretation and interpretations of the current paper coincide on interpretable topological spaces and interpretable Borel spaces, and this is proved in Theorem~\ref{fremlintheorem}. The current treatment has the advantage of avoiding the forcing relation altogether. As a result, the theory is easier to develop and understand, and there is a number of central results with no counterpart in Fremlin's work, such as Theorems~\ref{openmappingtheorem} (interpretation of an open continuous map is open continuous), ~\ref{absolutenesstheorem} (the interpretation map is $\Pi_1$-elementary embedding between topological structures), ~\ref{faithfulnesstheorem} (interpreting through an intermediate model is the same as interpreting with one step), or ~\ref{hardbanachtheorem} (regarding interpretations of Banach spaces and their duals).

The terminology of the paper follows the set theoretic standard of \cite{jech:set}. A tree is a partial ordering $\langle T, \leq\rangle$ such that for every $t\in T$ the set $\{s\in T\colon s\geq t\}$ is finite and linearly ordered by $\leq$. All trees are assumed to have a largest element, denoted by $0$. A branch $b$ through the tree $T$ is an inclusion-maximal linearly ordered subset of $T$, and for every $n\in\gw$ I will write $b\restriction n$ for the unique element $t\in b$ (if it exists) such that  the set $\{s\in T\colon s\geq t\}$ has size $n$. One piece of parlance is used constantly. Suppose that $M$ is a model of set theory and $M\models \langle X, \tau\rangle$ is a topological space.
Suppose that $\langle\hat X, \hat\tau\rangle$ is a topological space. If $\pi\colon X\to\hat X$ is a map, I say that $\pi$ extends to an interpretation if there is a way of defining $\pi$ for all open sets such that the resulting map is a topological interpretation of $\langle X, \tau\rangle$ to $\langle\hat X, \hat\tau\rangle$. The extension is unique as described in Definition~\ref{extensiondefinition}.

\section{Breakdown of results}

The paper is quite long and contains many results. To simplify navigation for the reader, I include the list of main results in an instructive order in this section. The starting point is the proof of existence of interpretations of topological spaces and of continuous functions between them.

\begin{theorem}
\textnormal{(Theorem~\ref{existencetheorem} simplified)}
Interpretations of regular Hausdorff spaces exist, they are unique, and they are regular Hausdorff again.
\textnormal{(Theorem~\ref{functiontheorem} simplified)}
Interpretations of continuous functions between regular Hausdorff spaces exist and they are unique.
\end{theorem}

It should be stressed that for the general category of regular Hausdorff spaces, the interpretation functor is fairly poorly behaved: injective functions may cease to be injective, interpretations may fail to commute with product and so on. Many examples of pathologies are provided throughout the paper. It is natural to immediately restrict to various subcategories of interpretable spaces. It turns out that if a space has a certain completeness feature then the feature typically survives the interpretation process. 

\begin{theorem}
The following spaces are interpreted as spaces in the same category:

\begin{enumerate}
\item \textnormal{(Corollary~\ref{compact1corollary} simplified)} compact Hausdorff spaces;
\item \textnormal{(Corollary~\ref{metrizablecorollary} simplified)} completely metrizable spaces;
\item \textnormal{(Corollary~\ref{uniformcorollary})} complete uniform spaces as long as the uniformity consists of countably many covers;
\item \textnormal{(Corollary~\ref{cechcorollary})} {\v C}ech complete spaces;
\item \textnormal{(Corollary~\ref{interpretablecorollary})} interpretable spaces.
\end{enumerate}
\end{theorem}

\noindent The method of proof also provides for a testable criterion (Proposition~\ref{testproposition}) as to whether a given map into a ``complete'' space extends to an interpretation or not.
 Next, I spend a great deal of effort to show that the interpretation functor commutes with natural operations on topological spaces. 

\begin{theorem}
In the category of interpretable spaces, the interpretation functor commutes with the following operations:

\begin{enumerate}
\item \textnormal{(Corollary~\ref{subspacecorollary} simplified)} taking a closed or $G_\gd$ subspace;
\item \textnormal{(Theorem~\ref{compactproducttheorem} simplified)} product of compact Hausdorff spaces of any size;
\item \textnormal{(Theorem~\ref{countableproducttheorem} simplified)} product of countably many spaces;
\item \textnormal{(Corollary~\ref{openquotientcorollary} simplified)} quotients modulo an open equivalence relation;
\item \textnormal{(Corollary~\ref{perfectquotientcorollary} simplified)} quotients modulo a perfect equivalence relation;
\item \textnormal{(Theorem~\ref{cktheorem} simplified)} the $C(X,Y)$ operation where $X$ is compact Hausdorff, $Y$ is completely metrizable, and the topology is the compact-open one;
\item \textnormal{(Theorem~\ref{hyperspacetheorem} simplified)} the hyperspace operation on interpretable spaces. Here, the hyperspace is understood to consist of nonempty compact sets and the topology is Vietoris.
\end{enumerate}
\end{theorem}

As soon as one steps out of the interpretable category, the commutativity features begin failing. Thus, the commutation with product may fail for Baire space times the space of wellfounded trees or the product of Sorgenfrey line with itself. The commutation with uncountable product will fail for $\gw^{\gw_1}$.

It is interesting to see how the interpretation functor acts on various topological structures. One common circumstance is that various predicates on such structures are Borel, and a theorem to the effect that they are interpreted faithfully is needed:

\begin{theorem}
\textnormal{(Theorem~\ref{boreltheorem})}
Let $M$ be a transitive model of set theory and $M\models \langle X, \tau\rangle$ is an interpretable space, and $\mathcal{B}$ is the $\gs$-algebra of Borel subsets of $X$.
Let $\pi\colon X \to\hat X$ be an interpretation, and let $\hat{\mathcal{B}}$ be the $\gs$-algebra of Borel subsets of $\hat X$. There is a unique extension $\pi\colon\mathcal{B}\to\hat{\mathcal{B}}$ commuting with complements and countable unions and intersections in the model $M$.
\end{theorem}

Now, define an \emph{interpretable structure} to be a tuple $\mathfrak{X}=\langle X_i\colon i\in I, R_j\colon j\in J, f_k\colon k\in K\rangle$ such that $X_i$ are 
interpretable spaces, $R_j$ are Borel relations between the various spaces $X_i$ for various finite arities, and $f_k$ are continuous functions between the various spaces, with closed or $G_\gd$ domains and various finite arities. Interpretable structures can be naturally interpreted between transitive models of set theory and its extensions via the previous theorems. The main result in this direction is 

\begin{theorem}
\textnormal{(Analytic absoluteness, Theorem~\ref{absolutenesstheorem})} Suppose that $M$ is a transitive model of set theory and $M\models\mathfrak{X}$ is an interpretable structure. Let $\pi\colon\mathfrak{X}\to\hat{\mathfrak{X}}$ be an interpretation. Then $\pi$ is a $\gS_1$-elementary embedding.
\end{theorem}

In particular, structures such as topological groups and semigroups,their continuous actions,  Banach spaces, $C^*$ algebras are interpreted faithully as structures of the same class, as long as their functions are continuous and their relations are Borel, the axiomatization of their class consists of $\gS_1$ and $\Pi_1$ sentences, and their domains are interpretable. I also prove a version of Shoenfield absoluteness, Theorem~\ref{shoenfieldabsoluteness}, which shows that in many cases, if a $\Pi_1$ formula defines a closed set $C$ in a structure $\mathfrak{X}$, then it defines the interpretation of the closed set $C$ in the interpreted structure $\hat{\mathfrak{X}}$. 

The main source of topological structures is functional analysis. There, I prove for example

\begin{theorem}
\textnormal{(Theorem~\ref{weakstartheorem} simplified)}
The interpretation of the unit ball in the dual space of a normed Banach space $X$ with weak${}^*$ topology is the unit ball with the weak${}^*$ topology of the dual space of the interpretation of $X$.
\end{theorem}

\begin{theorem}
\textnormal{(Theorem~\ref{hardbanachtheorem} simplified)}
The interpretation of the normed dual of a uniformly convex Banach space $X$ is the normed dual of the interpretation of $X$.
\end{theorem}

\begin{question}
Is an interpretation of a reflexive space always reflexive? Is the normed dual of a reflexive space always interpreted as the normed dual of the interpretation?
\end{question}

\begin{question}
\textnormal{(Ilijas Farah)} Is the interpretation of a simplex again a simplex?
\end{question}

An important feature of interpretations is that they behave in a predictable way if more models of set theory are present:

\begin{theorem}
Suppose that $M_0\subset M_1$ are transitive models of set theory, $M_0\models \langle X_0, {\tau}_0\rangle$ is an interpretable space, $M_1\models \pi_0\colon\langle X_0, {\tau}_0\rangle\to \langle X_1, {\tau}_1\rangle$ is an interpretation over $M_0$, and $\pi_1\colon\langle X_1, {\tau}_1\rangle\to\langle X_2, {\tau}_2\rangle$ is an interpretation over $M_1$. Then $\pi_1\circ\pi_0\colon\langle X_0, {\tau}_0\rangle\to\langle M_2, {\tau}_2\rangle$ is an interpretation over $M_0$.
\end{theorem}

\noindent The conclusion of this theorem may fail for such spaces as $\gw^{\gw_1}$ or $\mathbb{R}^{\mathbb{R}}$.

\begin{theorem}
Suppose that $X$ is an interpretable space, and $M$ is an elementary submodel of some large structure containing
$X$ as an element and some basis of $X$ as an element and a subset. Then the identity map from $X\cap M$ to $X$
can be extended to an interpretation of $X^M$ to $X$.
\end{theorem}

\noindent The conclusion fails for every non-Polish second countable space $X$ and countable model $M$.

As a final note, I list the regularity properties of topological spaces which are preserved under interpretations.

\begin{theorem}
\textnormal{(Corollary~\ref{preservationcorollary})}
The following properties of interpretable spaces are preserved under interpretations:

\begin{enumerate}
\item local Lindel{\"o}fness;
\item local connectedness;
\item local paracompactness;
\item local metacompactness.
\end{enumerate}
\end{theorem}

In Section~\ref{borelsection} I show how the interpretation functor can be canonically extended to topological spaces which are not interpretable, such as $C_p(\mathbb{R})$. Instead of viewing them as pure topological spaces, equip them with a Borel structure and then ask for the interpretations to respect the countable union and intersection operations on Borel sets. The basic features of the theory of interpretations of Borel spaces are parallel to pure topological spaces. I also show that in the class of proper bounding forcing extensions, this notion essentially coincides with interpretations of pure topological spaces.

\begin{theorem}
\textnormal{(Theorem~\ref{boundingtheorem} simplified)}
Suppose that $V$ is a proper bounding extension of $V$.
Suppose $M\models\langle X, \tau, \mathcal{B}\rangle$ is a regular Hausdorff space with a Borel $\gs$-algebra. If $\pi\colon \langle X, \tau\rangle\to\langle \hat X, \hat\tau\rangle$ is a interpretation of the topological space then $\pi$ extends to an interpretation of the Borel space.
\end{theorem}

Finally, in Section~\ref{fremlinsection} I prove that the interpretations introduced by Fremlin \cite{fremlin:top} coincide with the interpretations introduced in the present paper in the special case when $V$ is a generic extension of $M$ in the appropriate categories of spaces.

\section{The category of interpretable spaces}

The basic stepping stone for the interpretation theory developed in this paper is the category of
{\v C}ech complete spaces:

\begin{definition}
A topological space $X$ is \emph{{\v C}ech complete} if it is a $G_\gd$ subspace of a compact Hausdorff space.
\end{definition}

\noindent Thus, every compact Hausdorff space is {\v C}ech complete. So is every locally compact space (as its Alexandroff compactification is a Hausdorff space) and every completely metrizable space (by Fact~\ref{littlefact} below). A classical internal characterization of the class of {\v C}ech complete spaces will be useful:

\begin{definition}
Let $X$ be a topological space. A \emph{complete sequence of covers} for $X$ is a sequence $\langle C_n\colon n\in\gw\rangle$ of open covers such that for every collection $F$ of closed subsets of $X$ which has the finite intersection property and for every $n\in\gw$ contains a subset of ${\mathrm{cl}}(O_n)$ for some $O_n\in C_n$, the intersection $\bigcap F$ is nonempty.
\end{definition}

\begin{fact}
\label{littlefact}
\textnormal{\cite[Theorem 3.9.2]{engelking:general}} A regular Hausdorff space is {\v C}ech complete if and only if it has a complete sequence of covers.
\end{fact}

\noindent The category in which the interpretation functor is at its most natural is a slight extension of the category of {\v C}ech complete spaces:

\begin{definition}
A topological space $X$ is \emph{interpretable} if it is regular Hausdorff and an open continuous image of a {\v C}ech complete space.
\end{definition}

\noindent Thus, every {\v C}ech complete space is interpretable. The class of interpretable spaces was investigated under various names (sieve complete, $\lambda_b$) in early 70's \cite{chaber:monotonic, wicke:good, wicke:open} and a useful internal characterization was provided:

\begin{definition}
Let $X$ be a topological space. A \emph{sieve} is a tuple $\langle S, O(s)\colon s\in S\rangle$ where $S$ is a tree with a largest node $0$, $O(0)=X$, and for every $s\in S$, $O(s)\subset X$ is an open set and $O(s)=\bigcup\{O(t)\colon t\in S$ is an immediate successor of $s\}$. A sieve $S$ is \emph{complete} if for every infinite path $b\subset S$ and every collection $F$ of closed subsets of $X$ which has the finite intersection property and for every $n\in\gw$ contains a subset of ${\mathrm{cl}}(O(b\restriction n))$ for some $O_n\in C_n$, the intersection $\bigcap F$ is nonempty. The sieve is \emph{strong} if $\bar O(t)\subset O(s)$ whenever $t$ is an immediate successor of $s$. The sieve is \emph{finitely additive} if for every $s\in S$, the set $\{O(t)\colon t$ is an immediate successor of $s\}$ is closed under finite unions. 
\end{definition}

\begin{fact}
\label{sieveproposition}
Let $X$ be a regular Hausdorff space.

\begin{enumerate}
\item \textnormal{\cite{wicke:good}} $X$ is interpretable if and only if it has a complete sieve;
\item \textnormal{\cite[Lemma 2.3 and 2.4]{michael:triquotient}} $X$ is interpretable if and only if it has a sieve which is complete, strong, and finitely additive;
\item \textnormal{\cite[Lemma 2.5]{michael:triquotient}} A strong sieve $\langle S, O(s)\colon s\in S\rangle$ is complete if and only if whenever $b\subset S$ is a branch then the set $K=\bigcap_n O(b(n))$ is compact, and for every open set $O\subset X$, if $K\subset O$ then for some $n\in\gw$ $O(b(n))\subset O$ holds.
\end{enumerate}
\end{fact}

\noindent The extent of the class of interpretable spaces is best appreciated from the perspective of the following theorem, which with the exception of the hyperspace operation is contained in \cite{michael:triquotient}.

\begin{theorem}
\label{interpretabletheorem}
The class of interpretable spaces is closed under the following operations:

\begin{enumerate}
\item a closed subset;
\item a $G_\gd$ subset;
\item countable product;
\item hyperspace of compact sets with the Vietoris topology;
\item perfect and open continuous images as long as they are regular Hausdorff.
\end{enumerate}

\noindent A space is locally interpretable if and only if it is locally interpretable.
\end{theorem}

\begin{proof}
For the first two items, let $X$ be a compact Hausdorff space, $Y\subset X$ its $G_\gd$ subset, and $f\colon Y\to Z$ be an open continuous surjection onto a regular Hausdorff space. If $C\subset Z$ is closed, then $D=f^{-1}C\subset X$ is a relatively closed set, $f\restriction D\colon D\to C$ is open, and $D$ is a $G_\gd$ subset of its closure in the space $X$;
thus, $C$ is interpretable. The case of $G_\gd$ set $C\subset Z$ is identical.

For (3), suppose that $X_n$ for $n\in\gw$ are interpretable spaces, with a complete strong sieve $\mathcal{S}_n=\langle S_n, O_n(s)\colon s\in S_n\rangle$ on each. Consider the product sieve $\mathcal{T}=\langle T, P(t)\colon t\in T\rangle$ on the space $Y=\prod_nX_n$
defined as follows. A node of the tree $T$ is a tuple $t=\langle s_n\colon n\in m\rangle$ for some $m\in\gw$ such that each $s_n$ is an element of $m$-th level of the tree $S_n$. The ordering on the tree $T$ is defined coordinatewise. 
$t=\langle s_n\colon n\in m\rangle\in T$ define $P(t)=\{y\in Y\colon\forall n\in m\ y(n)\in O_n(s_n)\}$. Clearly, $\mathcal{T}$ is a strong sieve on $Y$; I will use Fact~\ref{sieveproposition}(3) to verify that $\mathcal{T}$ is complete.

Suppose that $b\subset T$ is an infinite branch. Then, for each $n\in\gw$, the nodes of $S_n$ mentioned in the nodes of $b$ form an infinite branch $b_n\subset S_n$. Let $K_n=\bigcap_m O_n(b_n(m))$; $K_n\subset X_n$ is a compact set by the completeness of the sieve $\mathcal{S}_n$. It is immediate that $\bigcap_mP(b(m))=\prod_nK_n\subset Y$ is compact. Now suppose that $U\subset Y$ is an open set containing $\prod_nK_n$. By a compactness argument, there is a number $m\in\gw$ and open sets $U_n\subset X_n$
for $n\in m$ such that $\prod_{n\in m}U_n\times\prod_{n\geq m}X_n\subset U$, and $U_n$ contains $K_n$ for all $n\in m$.
The completeness of the sieves $\mathcal{S}_n$ for $n\in m$ shows that there is a number $k\in\gw$ such that $O_n(b_n(k))\subset O_n$ for all $n\in m$.
Then, $P(b(k))\subset U$ as desired in Fact~\ref{sieveproposition}(3).

For (4), if $X$ is a topological space, write $K(X)$ of the hyperspace of its nonempty compact subsets with 
Vietoris topology. Suppose that $X$ is interpretable, with a complete finitely additive strong sieve $\mathcal{S}=\langle S, O(s)\colon s\in S\rangle$ on it. Consider the strong sieve $\mathcal{T}=\langle S, P(s)\colon s\in S\rangle$ on $K(X)$, where $P(s)=\{K\in K(X)\colon K\subset O(s)\}$; this is indeed a sieve by the finite additivity of $\mathcal{S}$ and a  compactness argument. I will use Fact~\ref{sieveproposition}(3) to show that $\mathcal{T}$ is complete on $K(X)$. 

Suppose that $b\subset S$ is an infinite branch. Write $Y=\bigcap_nO(b\restriction n)\subset X$; by the completeness of the sieve $\mathcal{S}$, this is a compact subset of $X$.  It is immediate to see that $\bigcap_nP(b\restriction n)=K(Y)\subset K(X)$ and so this set is compact. Suppose that
$U\subset K(X)$ is an open set containing $K(Y)$. A compactness argument shows that there must be an open set $O\subset X$ such that $Y\subset O$ and $\{K\in K(X)\colon K\subset O\}\subset U$. The completeness of the sieve $\mathcal{S}$ shows that there is an $n\in\gw$ such that $O(b\restriction n)\subset O$. Then, $P(b\restriction n)\subset U$ and the completeness of the sieve $\mathcal{T}$ follows by Fact~\ref{sieveproposition}(3).

For (5), the case of open images follows straight from the definitions. For the perfect case, note that if $f\colon X\to Y$ is a perfect surjection then the set $C=\{K\in K(X)\colon f\restriction K$ is constant$\}$
is closed. Chase diagrams to show that the map $g\colon C\to Y$ defined by $g(K)=$the unique element of $f''K$ is continuous and open. Then, use (1, 3, 4) to conclude that $Y$ is an interpretable space.

 For the last sentence, suppose that a space $X$ is locally interpretable. For every point $x\in X$ find an open set $O_x\subset X$ and a continuous open surjection $f_x\colon Y_x\to O_x$ from a \v Cech complete space $Y_x$ onto $O_x$. Let $Y$ be the topological sum of the spaces $Y_x$ and $f\colon Y\to X$ be the sum of the mappings $f_x$ for $x\in X$. It is not difficult to check that $Y$ is {\v C}ech complete and the function $f$ is a continuous open surjection from $Y$ to $X$, confirming that $X$ is interpretable.
\end{proof}

Certain features of the category of interpretable spaces can be verified using the methods of this paper even though they do not refer to any models of set theory at all:

\begin{theorem}
\label{zztheorem}
\begin{enumerate}
\item Every interpretable space $X$ is a continuous open image of a $G_\gd$ subspace of a compact space $Z$ such that the weight of $Z$ is no larger than the weight of $X$.
\item Every second countable interpretable space is Polish.
\end{enumerate}
\end{theorem}

\begin{proof}
Let $X$ be an interpretable space, and fix a compact space $Z'$, its $G_\gd$ subspace $Y'$, and a continuous open surjection $f'\colon Y'\to X$. Let $\kappa$ be the weight of $X$; without loss of generality $\kappa$ is infinite. Let $\lambda$ be a regular cardinal such that $H_\kappa$ contains all objects named so far. Let $M\prec H_\lambda$ be an elementary submodel of size $\kappa$ which contains all objects named so far, contains a basis of $X$ as a subset, and a complete sieve on $X$ as an element. Let $\bar M$ be the transitive collapse of the model $M$, and $\bar X, \bar Y', \bar Z', \bar f'$ the images of $X, Y', Z', f'$ under the transitive collapse map. Note that the inverse of the collapse map from $\bar X$ to $X$ is an interpretation of $\bar X$ by Theorem~\ref{submodeltheorem}. Now, let $Y, Z$, and $f\colon Y\to X$ be interpretations of $\bar Y', \bar Z'$ and $\bar f'$ respectively. Corollary~\ref{compact1corollary} shows that $Z$ is compact; its weight is at most $\kappa$ since the interpretations of open subsets of $\bar Z'$ in the model $\bar M$ generate its topology. $Y$ is a $G_\gd$ subset of $Z$ by Corollary~\ref{subspacecorollary} and the map $f$ is a continuous open surjection by Theorem~\ref{openmappingtheorem}. This completes the proof of (1).

For (2), let $X$ be a second countable interpretable space. Use (1) to find a second countable compact space $Z$, its $G_\gd$ subset $Y$, and a continuous open surjection $f\colon Y\to Z$. By the Urysohn metrization theorem, all three spaces are metrizable. In particular, $Z$ is compact Polish and $Y$ as a $G_\gd$ subset of a Polish space is Polish as well. By a theorem of Sierpinski, every metrizable continuous open image of a Polish space is Polish, so $X$ is Polish as desired.
\end{proof}

\noindent Note that there are many interesting spaces which are not interpretable.

\begin{example}
The following spaces are not interpretable:

\begin{enumerate}
\item any non-$G_\gd$-subset of a Polish space;
\item $C_p(\mathbb{R})$, the space of continuous functions from $\mathbb{R}$ to $\mathbb{R}$ with the pointwise convergence topology;
\item $\mathbb{R}^{\mathbb{R}}$ with the pointwise convergence topology;
\item the Sorgenfrey line.
\end{enumerate}
\end{example}

\begin{proof}
For (1), note that a subset of a Polish space is second countable, and so to be interpretable, by Theorem~\ref{zztheorem}(2) it must be Polish. It is well-known that a subset of a Polish space is Polish in the subspace topology if and only if it is $G_\gd$. For the other items, the proof of non-interpretability can be in fact derived from some interpretation pathologies that cannot occur in interpretable spaces. For (2), Example~\ref{cpexample} shows that whenever an unbounded real is added, the interpretation of $C_p(\mathbb{R})$ in the generic extension cannot be extended to the $\gs$-algebra of Borel sets; this does not happen for interpretable spaces by Theorem~\ref{boreltheorem}.  For (3), the space $\mathbb{R}^{\mathbb{R}}$ contains a closed copy of $\gw^{\gw_1}$ with the product topology. The space $\gw^{\gw_1}$ is not interpretable since there is a two-step forcing extension $V\subset V[G]\subset V[H]$ such that interpreting first in $V[G]$ and then in $V[H]$ does not yield the same result as interpreting $\gw^{\gw_1}$ in $V[H]$--Example~\ref{productexample}; such a thing cannot happen for interpretable spaces by Theorem~\ref{faithfulnesstheorem}. For (4), whenever a forcing extension adds a new real then the interpretation of the product of two Sorgenfrey lines is not the product of two copies of the interpretation of the Sorgenfrey line by Example~\ref{sorgenfreyexample}; again, this does not occur for interpretable spaces by Theorem~\ref{countableproducttheorem}.
\end{proof}

There are important distinctions in the above examples. (2) is naturally viewed as a Borel subspace of an interpretable space, and for such cases I will find a convenient way of interpreting it as a Borel topological structure. (3) and (4) exhibit important pathologies which most likely prevent any attempt at incorporating them in a general interpretation theory. The third case may behave well if a certain common class of models is considered; the fourth case is probably hopeless in all but the most trivial circumstances.

\section{The existence theorem}

The first order of business is to show that topological and Borel-topological interpretations exist for a large class of spaces.

\begin{theorem}
\label{existencetheorem}
Suppose that $M$ is a model of set theory and $M\models \langle X, \tau\rangle$ is a regular Hausdorff space. Then the interpretation of $\langle X, \tau\rangle$ exists, it is unique up to interpretation equivalence, and it is regular Hausdorff.
\end{theorem}

\begin{proof}
First note that some preinterpretations indeed exist--the identity map on $X$ is one of them. Note also that the target spaces of topological preinterpretations of $X$ must be regular Hausdorff, since the property is expressible in terms of unions and intersections of open sets: for every $O\in\tau$, $O=\bigcup_{i\in I}O_i$ such that for every $i\in I$ ${\mathrm{cl}}(O_i)\subset O$ holds, and the last statement can be expressed as $\exists P_i\in\tau\ O_i\cap P_i=0$ and $O\cup P_i=X$.

Consider the set $Y$ of all sets $A\subset \tau$ such that $A$ is a maximal ideal of open sets: it contains no finite subcover and  is maximal with respect to inclusion. Let $\chi\colon X\to Y$ be the map defined by $\chi(x)=\{O\in\tau\colon x\notin O\}$. Equip the space $Y$ with the topology generated by sets $\chi(O)=\{A\in Y\colon O\notin A\}$ for $O\in\tau$.  The space $Y$ is close to the Wallman extension of $X$, and in fact equal to it in the trivial case when $M=V$. The map $\chi$ is most likely not a preinterpretation, but it is nevertheless universal in the following sense:

\begin{claim}
Suppose that $\pi\colon\langle X, \tau\rangle\to\langle \hat X, \hat\tau\rangle$ is a topological preinterpretation.
Then there is a unique map $h_\pi\colon\hat X\to Y$ such that $\chi= h\circ\pi$ and for every $O\in\tau$, $h^{-1}\chi(O)=\pi(O)$ holds.
\end{claim}

\begin{proof}
Define a function $h_\pi=h\colon\hat X\to Y$ by $h(x)=\{O\in\tau\colon x\notin\pi(O)\}$. Use the regular Hausdorff assumption to show that the value of function of $h$ are in fact elements of the space $Y$; once this is done, the verification of the requested properties of the map $h$ is trivial.

Since $\pi$ is a preinterpretation, the set $h(x)\subset\tau$ cannot contain any finite subcover of $X$: such a subcover would be an element of $M$ and its $\pi$-image would have to cover all of $\hat X$. To show that $h(x)$ is inclusion-maximal, for every $O\in\tau$ such that $x\in\pi(O)$ I must produce a set $P\in h(x)$ such that $O\cup P=X$.

By the regular Hausdorffness of the space $X$ in the model $M$, $M\models O=\bigcup_iO_i$ where the closure of each $O_i$ is a subset of $O$, or in other words, there is $P_i\in\tau$ such that $O_i\cap P_i=0$ and $O\cup P_i=X$. Since $\pi$ is a preinterpretation, there is an index $i$ such that $x\in\pi(O_i)$. Then $\pi(P_i)\cap \pi(O_i)=0$, in particular $P_i\in h(x)$
and $P_i\cup O=X$ as desired.
\end{proof}

Now, let $\hat X_0=\bigcup\{\rng(h_\pi)\colon \pi$ is a topological preinterpretation of $X\}\subset Y$ and equip the set $\hat X_0$ with the topology inherited from $Y$. Let $\pi_0\colon X\to\hat X_0$ be defined by $\pi_0(x)=\chi(x)$, and let $\pi_0\colon\tau\to\power(\hat X_0)$ be defined by $\pi_0(O)=\chi(O)\cap\hat X_0=\bigcup\{h_\pi''\pi(O)\colon \pi$ is a topological preinterpretation of $X\}$. It is easy to verify that $\pi_0$ is a topological preinterpretation of $X$ to which every preinterpretation
$\pi$ is reducible via the map $h_\pi$; i.e., $\pi_0$ is the topological interpretation of $X$.

For the uniqueness of the interpretation, suppose for example that $\pi\colon\langle X, \tau\rangle\to\langle\hat X, \hat\tau\rangle$ is another topological interpretation. There must be a reduction $h\colon \hat X_0\to\hat X$ of $\pi_0$ to $\pi$.
It is easy to observe that $h_\pi\circ h$ must be the identity and therefore the reduction $h$ in fact witnesses the equivalence of the two interpretations.
\end{proof}

\begin{example}
\label{isolatedexample}
Suppose that $X$ is a topological space and $P$ is the poset of its nonempty open sets ordered by inclusion. If $G\subset P$ is a generic filter, in the generic extension $V[G]$ the map $\pi\colon X\to X\cup\{G\}$ where points of $X$ are mapped to themselves and ground
model open sets $O\subset X$ are mapped to $O\cup\{G\}$ if $G\in O$ and to $O$ otherwise, is a preinterpretation by a genericity argument. Thus, for every regular Hausdorff space $X$ which contains no isolated points there is a generic extension in which the interpretation map of $X$ is not a surjection.
\end{example}

This means that for example the space $\mathbb{Q}$ will have have rather counterintuitive interpretations in generic extensions containing Cohen or even unbounded reals. Note that the space $\mathbb{Q}$ does not belong to the interpretable category.

\begin{example}
Suppose that $M\models\langle X, \tau\rangle$ is a space whose topology is generated by a metric $d$. If $\pi\colon X\to\hat X$ is an interpretation, one can define a metric $\hat d$ on $\hat X$ by setting $\hat d(x_0, x_1)\leq\eps$ if $x_0, x_1$ belong to $\pi(O_0)$, $\pi(O_1)$
respectively for some open sets $O_0, O_1\in\tau$ such that $y_0\in O_0$ and $y_1\in O_1$ imply $d(y_0, y_0)<\eps$. It is not difficult to see that $\hat d$ is a metric generating the topology of $\hat X$, $\pi''X\subset \hat X$ is dense, and $\hat d\circ \pi=d$.
This means that the interpretation of a metrizable space is again metrizable and for example the interpretation of $\mathbb{Q}$ can be viewed as a set of reals.
\end{example}

\begin{example}
Suppose that $M\models\langle X, \tau\rangle$ is a uniform space, with the topology generated by a uniform set $\Theta$ of covers. Let $\pi\colon \langle X, \tau\rangle\to\langle \hat X, \hat \tau\rangle$ be an interpretation. The space $X$ is then uniform,
and $\hat\Theta=\{\pi''C\colon C\in\Theta\}$ is a uniform set of covers generating its topology. To see this, note that if $C,D\in M$ are open covers such that $C$ is a star-refinement of $D$ (for every $O\in C$ there is $P\in D$ such that $\phi(O, P)=\bigcup\{Q\in C\colon Q\cap O\neq 0\}\subset P$ holds) then $\pi''C$ is a star-refinement of $\pi''C$ (since $\phi(O,P)$ is equivalent to $\phi(\pi(O), \pi(P))$). 
\end{example}

Intepretations of regular Hausdorff spaces commute with taking open or closed subspace. This is the contents of the following theorem. 

\begin{theorem}
\label{subspacetheorem}
Suppose that $M$ is a model of set theory and $M\models\langle X, \tau\rangle$ is a regular Hausdorff space and $O\in\tau$ is an open set. Let $\pi\colon X\to\hat X$ be an interpretation.

\begin{enumerate}
\item $\pi\restriction O$, the map from $X\restriction O$ to $X\restriction\pi(O)$, extends to an intepretation of the space $X\restriction O$.
\item $\pi\restriction X\setminus O$, the map from $X\setminus O$ to $\hat X\setminus\pi(O)$, extends to an intepretation of the space $X\setminus O$
\end{enumerate}
\end{theorem}

\begin{proof}
I will work on (2), (1) is similar. Write $Y=X\setminus O$ and $\hat Y=\hat X\setminus\pi(O)$, both equipped with the inherited topologies $\gs$ and $\hat\gs$ respectively. Define a map $\chi\colon Y\to \hat Y$ by $\chi(y)=\pi(y)$, and a map
$\chi\colon\gs\to\hat\gs$ by $\chi(P\cap Y)=\pi(P)\cap\hat Y$ for a set $P\in\tau$. Note that this depends only on $P\cap Y$ and not on $P$, since if $P_0, P_1\in\tau$ are sets with the same intersection with $Y$, then $P_0\cup O=P_1\cup O$,
so $\pi(P_0)\cup\pi(O)=\pi(P_0)\cup\pi(O)$ and so $\pi(P_0)$ and $\pi(P_1)$ have the same intersection with $\hat Y$. Since $\pi$ is an interpretation of $X$, $\chi$ is a preinterpretation of $Y$.

To show that in fact $\chi$ is an interpretation of $Y$, suppose that $\chi'\colon Y\to \hat Y'$ is another preinterpretation of $Y$; I must find a reduction of $\chi'$ to $\chi$. Consider the adjunction space $\hat X'$ which is the union of $X\cup\hat Y'$ modulo the equivalence induced by the attaching map $\chi'\colon Y\to\hat Y'$, with the resulting topology $\hat\tau'$. Consider the map $\pi'\colon X\to\hat X'$ given by $\pi'(x)=[x]$, and the map $\pi'\colon\tau\to\hat\tau'$ given by $\pi'(O)=[O]\cup\chi'(O\cap Y)$.
It is not difficult to check that $\pi'$ is a preinterpretation of the space $X$, and so is reducible to $\pi$ via some map $h\colon\hat X'\to\hat X$. Clearly, the map $h\restriction Y'$ reduces $\chi'$ to $\chi$ as desired.
\end{proof}

\begin{theorem}
\label{functiontheorem}
Suppose that $M$ is a model of set theory and $M\models \langle X, \tau\rangle, \langle Y, \gs\rangle$ are regular Hausdorff spaces and $f\colon X\to Y$ is a continuous function.
Suppose that $\pi\colon X\to\hat X$ and $\chi\colon Y\to\hat Y$ are interpretations.

\begin{enumerate}
\item There is a unique continuous function $\hat f\colon \hat X\to\hat Y$ which contains the set $\{\langle\pi(x), \chi(y)\rangle\colon x\in X, y\in Y, f(x)=y\}$ as a subfunction.
\item For the unique function $\hat f$, whenever $O\in\tau$ and $P\in\gs$ are open sets and $f^{-1}P=O$, then
$\hat f^{-1}\chi(P)=\pi(O)$.
\end{enumerate}
\end{theorem}

\begin{proof}
First of all, every continuous function satisfying (1) has to satisfy (2) as well. Suppose that $P\in\gs$ and $O\in\tau$ are sets such that $f^{-1}P=O$. First, suppose for contradiction that $x\in \pi(O)$ and $\hat f(x)\notin\chi (P)$. Since $\pi$ is a preinterpretation, there must be sets $\bar P\in\gs$ and $\bar O\in\tau$ such that $f^{-1}P=O$, ${\mathrm{cl}}(\bar P)\subset P$
and $x\in\pi(\bar O)$. Then, the $\hat f$-preimage of $\chi(X\setminus{\mathrm{cl}}(\bar P))$ contains $x$ but no points in $\pi''\bar O$
which are dense around $x$, contradicting the continuity of the function $\hat f$ at $x$. The contradiction in the case that
$x\notin\pi(O)$ and $f(x)\in\chi(P)$ is obtained in a similar way.

The existence and uniqueness of the function $\hat f$ immediately follows from the following claim.

\begin{claim}
For every $x\in\hat X$ there is a unique $y\in\hat Y$ such that whenever $O\in\tau$ and $P\in\gs$ are open sets and $f^{-1}P=O$, then $x\in\pi(O)\liff y\in\chi(P)$.
\end{claim}

\begin{proof}
The uniqueness of the point $y$ is clear: whenever $y_0\neq y_1\in\hat Y$ are distinct points then there are disjoint sets $P_0, P_1\in\tau$ such that $x_0\in\chi(P_0)$ and $x_1\in\chi(P_1)$. Let $O_0=f^{-1}P_0$ and $O_1=f^{-1}P_1$.
These are disjoint open subset of $X$, and so $x$ can belong to at most one of $\pi(O_0)$ and $\pi(O_1)$.

For the existence of the point $y$, consider the space $Z=Y\cup\{p\}$ for some point $p$, and define a map $\xi\colon Y\to Z$ by $\xi(y)=y$. Define also the map $\xi\colon\gs\to\power(Z)$ by $\xi(P)=P\cup\{p\}$ if $x\in\pi(f^{-1}(P)$, and $\xi(P)=P$ otherwise. It is easy to use the fact that $\pi$ is a preinterpretation to show that the map $\xi$ commutes with finite intersections and unions in the model $M$. Thus, equipping the set $Z$ with the topology generated by the range of $\xi$, the map $\xi$ turns into a preinterpretation of the space $Y$. Since $\chi$ is an interpretation,
there is a reduction $h\colon Z\to\hat Y$ of $\xi$ to $\chi$. It is easy to check that the point $y=h(p)\in\hat Y$ works as desired.
\end{proof}

Define the function $\hat f\colon \hat X\to\hat Y$ by letting $\hat f(x)=y$ if $x,y$ satisfy the statement of the claim. It is clear that $\hat f$ satisfies (2). Since the range of $\chi$ generates the topology on the space $\hat Y$, it is clear that $\hat f$ is continuous. This completes the proof of the theorem.
\end{proof}

\begin{corollary}
Suppose that $M\models X, Y, Z$ are regular Hausdorff spaces and $f\colon X\to Y$ and $g\colon Y\to Z$ are continuous functions. Then the interpretation of $g\circ f$ equals the composition of interpretations of $f$ and $g$.
\end{corollary}

\begin{proof}
The composition of the interpretations is continuous, and it contains the pointwise image of $g\circ f$ under the interpretation maps as a subset. By the uniqueness part of Theorem~\ref{functiontheorem}, it must be equal to the interpretation of $g\circ f$.
\end{proof}

\begin{corollary}
Suppose that $M\models X, Y$ are regular Hausdorff spaces and $f\colon X\to Y$ is a continuous function. Then $y\in Y$ is in the range of $f$ if and only if the interpretation of $y$ is in the range of the interpretation of $f$.
\end{corollary}

\begin{proof}
Let $\pi\colon X\to\hat X$ and $\chi\colon Y\to\hat Y$ and $\hat f\colon \hat X\to\hat Y$ be interpretations. If $f(x)=y$ then $\hat f(\pi(x))=\chi(y)$ by the definition of $\hat f$, proving the left-to-right direction. For the right-to-left direction, if $y\in Y$
is a point which does not belong to $\rng(f)$ then $X=\bigcup \{f^{-1}P\colon P\subset Y$ is open and $y\notin P\}$. The interpretations commute with the union, $f$-preimage, and membership and so every element of $\hat X$ is mapped into an open set which does not
contain $\pi(y)$ as desired.
\end{proof}

\begin{corollary}
\label{homeocorollary}
Interpretation of a homeomorphism is a homeomorphism.
\end{corollary}

\begin{example}
An interpretation of a covering map between two regular Hausdorff spaces is a covering map between the interpreted spaces.
Suppose that $M$ is a model of set theory and $M\models X,Y$ are regular Hausdorff spaces and $f\colon X\to Y$ is a covering map. By the definition of covering, in the model $M$ the set $C=\{O\subset Y\colon O$ is open and $f^{-1}O=\bigcup D_O$ for some nonempty family $D_O$ consisting of pairwise disjoint open subsets of $X$ on which $f$ is a homeomorphism with $O\}$ is an open cover of $Y$. Let $\pi\colon X\to\hat X$ and $\chi\colon Y\to\hat Y$ be Borel-topological interpretations. $\chi''C$ is still an open cover of the space $\hat Y$. Moreover, for each open set $O\in C$,
$\hat f^{-1}O=\pi(\bigcup D_O)=\bigcup\pi''D_O$. The family $\pi''D_O$ still consists of pairwise disjoint open subsets of $\hat X$. Moreover, for every set $P\in D_O$, the function $f\restriction P$ is interpreted as a homeomorphism between $\pi(P)$ and $\chi(O)$. This completes the verification that $\hat f\colon\hat X\to\hat Y$ is a covering map.
\end{example}

\begin{example}
An interpretation of an injection need not be an injection. Let $X, Y$ be two dense disjoint sets of $\mathbb{Q}$ with the inherited order topology and let $X\cup Y$ be their topological sum. Let $f\colon X\cup Y\to\mathbb{Q}$ be the identity map. Pass to a generic extension $V[G]$ in which there is a Cohen real $r\in\mathbb{R}$. The interpretation of $X\cup Y$ can be viewed as a topological sum of two sets of reals, both of which contain $r$. The two copies of $r$ will be mapped by $\hat f$ to the same value, namely $r$ itself.
\end{example}

\noindent This pathology will not occur in the class of interpretable spaces by Corollary~\ref{injectioncorollary}. 

\begin{example}
The interpretation of a surjection need not be a surjection. Let $f$ be the surjective identity map from $X=$the reals with discrete topology to $Y=$the reals with the usual Euclidean topology. Pass to any generic extension $V[G]$ which contains new reals.
 Easy computations exhibited elsewhere in the paper show that the interpretation of $X$ is just the space $\hat X$ of ground model reals with the discrete topology, $\hat Y$ is just the space of all reals with the Euclidean topology, and $\hat f$ is the identity map, which now is not surjective anymore.
\end{example}

\noindent This feature can hardly be called a pathology as both spaces involved are interpretable and very natural.
Surjectivity of the interpreted surjections, if it occurs at all, is an important issue. It is normally guaranteed by stronger properties of the maps, such as openness or perfectness as in Theorems~\ref{perfectmappingtheorem} and ~\ref{openmappingtheorem}.

\section{Interpretations of complete spaces}

It turns out that interpretations of spaces in natural completeness categories have strong uniqueness features which make it much easier to evaluate them.

\begin{definition}
Let $\langle X, \tau\rangle$ be a topological space. A triple $\langle U, \leq, f\rangle$ is \emph{a completeness system} if $U$ is a set, $\leq$ is a partial ordering on it, and $f\colon A\to\tau$ is a function so that

\begin{enumerate}
\item for every $u\in U$, $f(u)\subseteq\bigcup\{f(v)\colon v<u\}$ holds;
\item for every strictly descending sequence $\langle u_n\colon n\in\gw\rangle$ in the ordering $\leq$ and every collection $F$ of closed subsets of $X$ with the finite intersection property such that ${\mathrm{cl}}(f(u_n))$ contains an element of $F$ for every $n\in\gw$,
then $\bigcap F\neq 0$.
\end{enumerate}
\end{definition}

\begin{theorem}
\label{sievetheorem}
Suppose that $M$ is a transitive model of set theory and $M\models X$ is a regular Hausdorff space and $\langle U, \leq, f\rangle$ is a completeness system on $X$. 
There is a unique preintepretation $\pi\colon X\to\hat X$ for which $\langle U, \leq, \hat f\rangle$ is a completeness system on $\hat X$
and it is the interpretation.
\end{theorem}

\noindent Here, the symbol $\hat f$ denotes the function defined by $\hat f(u)=\pi(f(u))$. Note that since the interpretation is defined without regard to the completeness system, it follows that every completeness system on $X$ in the model $M$
is interpreted as a completeness system on $\hat X$.

\begin{proof}
The uniqueness part is easy.  Suppose that $\pi\colon X\to\hat X$ is a preinterpretation such that $\langle U, \leq, \hat f\rangle$ is a completeness system; it will be enough to show that $\pi$ is an interpretation. To this end, suppose that $\chi\colon X\to Y$ is another preinterpretation.
For every point $y\in Y$, use (1) to find a descending sequence $\langle u_n\colon n\in\gw\rangle$ in $U$ such that $y\in \bigcap_n\chi(f(u_n))$. Consider the collection $F_y=\{{\mathrm{cl}}(\pi(O))\colon O\in\tau, y\in\chi(O)\}$. This is a filter of closed subsets of $\hat X$
such that for every $n\in\gw$ ${\mathrm{cl}}(f(u_n))\in F$. By the completeness assumption on the space $\hat X$, $\bigcap F_y\neq 0$. It is not difficult to use Hausdorffness of $\hat X$ to show that the intersection can contain at most one point. Let $h\colon Y\to\hat  X$ be the map such that $h(y)\in\bigcap F_y$ for all $y\in Y$. It is immediate that $h$ is a reduction of $\chi$ to $\pi$.

The existence of the required preinterpretation is a little harder. Let $\hat X$ be the set of all collections $A\subset\tau$ which do not contain a finite subcover of $X$, are maximal with respect to that condition, and such that there is an infinite descending sequence $\langle u_n\colon n\in\gw\rangle$ such that for every $n\in\gw$, $X\setminus\bigcap_{m\in n}{\mathrm{cl}}(f(u_m))\in A$.
Equip $\hat X$ with the topology generated by the sets $\pi(O)=\{A\in\hat X\colon O\notin A\}$ for $O\in\tau$.
Let $\pi\colon X\to\hat X$ be the map defined by $\pi(x)=\{O\in\tau\colon x\notin O\}$. I will show that $\pi$ is a preinterpretation
and $\langle U, \leq, \hat f\rangle$ is a completeness system on $\hat X$.

To show that $\pi$ is a preinterpretation, suppose that $O=\bigcup_{i\in I}O_i$ is a union of open sets in the model $M$; I must show that $\pi(O)=\bigcup_i\pi(O_i)$. The right-to-left inclusion is clear from the definitions. For the left-to-right inclusion, suppose that
$A\in\hat X$ is a point and $A\in\pi(O)$, meaning that $O\notin A$ holds; I must find $i\in I$ such that $O_i\notin A$. Find a set $P\in\tau$ such that $P\in A$ and $O\cup P=X$. In the model $M$, consider the tree $T$ of all finite attempts to build a
descending sequence $\langle u_n\colon n\in\gw\rangle$ in $U$ such that for every $n\in\gw$, for no finite set $J\subset I$ it is the case that $\bigcap_{m\in n}{\mathrm{cl}}(f(u_m))\subset\bigcup_{i\in J}O_i\cup P$.  The tree $T$ is well-founded in the model $M$ since any descending sequence of this form would yield a set $F=\{{\mathrm{cl}}(f(n))\colon n\in\gw, X\setminus O_i\colon i\in I, X\setminus P\}$ with finite intersection property. The intersection $\bigcap F$ would be nonempty by the completeness of the system on $X$, and any point in that intersection would have to belong to $O\setminus\bigcup_iO_i$; a contradiction. Since the model $M$ is transitive, the tree $T$ is well-founded even in the model $V$. Use the definition of the set $\hat X$
to find a descending sequence $\langle u_n\colon n\in\gw\rangle$ such that $X\setminus\bigcap_{m\in n}{\mathrm{cl}}(f(u_m))\in A$ for every $n\in\gw$. The sequence does not form an infinite path through the tree $T$ and so there must be $n\in\gw$ and a finite set $J\subset I$ such that $\bigcup_{i\in J}O_i\cup P\cup\bigcup_{m\in n}(X\setminus{\mathrm{cl}}(f(m)))=X$. This means that one of the sets $O_i$ for $i\in J$ must fail to belong to $A$, in other words $A\in \pi(O_i)$ as desired.
\end{proof}

\begin{corollary}
\label{compact1corollary}
Suppose $M\models X$ is a compact Hausdorff space. Then $X$ has a unique compact topological preinterpretation which is also its topological interpretation.
\end{corollary}

\begin{proof}
Note that a space $X$ is compact if and only if $\gw$ with reverse ordering and the function $f$ defined by $f(n)=X$ for every $n\in\gw$ together form a completeness system.
\end{proof}

\begin{corollary}
\label{compactcorollary}
Interpretations of compact subsets of regular Hausdorff spaces are compact.
\end{corollary}

\begin{proof}
Suppose that $M$ is a transitive model of set theory and $M\models X$ is a regular Hausdorff space and $K\subset X$ is compact.
Let $\pi\colon X\to\hat X$ be an interpretation of $X$. Theorem~\ref{subspacetheorem} shows that $\pi\restriction K\colon K\to\pi(K)$
is an interpretation. Corollary~\ref{compact1corollary} then implies that $\pi(K)$ is compact.
\end{proof}

\begin{corollary}
\label{metrizablecorollary}
Every completely metrizable space $X$ with a metric $d$ in a transitive model $M$ is interpreted as a completely metrizable space. In fact, the interpretation is just the completion of $X$ with respect to $d$ with the natural interpretation map.
\end{corollary}

\begin{proof}
Work in the model $M$. Let $U$ be the set of all open balls of finite $d$-radius and let $v<u$ if radius of $v$ is smaller than half of the radius of $u$. The function $f$ on $U$ is defined by $f(u)=u$. This is a completeness system on $X$. Let $\pi\colon X\to\hat X$ be an interpretation. By Theorem~\ref{sievetheorem}, $\langle U, <,f\rangle$ is interpreted as a completeness system on the interpretation $\hat X$. Define a metric $\hat d$ on $\hat X$ by setting $\hat d(x,y)<\eps+2^{-n}$ just in case $x\in \pi(O)$ and $y\in\pi(P)$ for some open balls $O,P\subset X$ whose centers have distance $<\eps$
and whose radii are smaller than $2^{-n-1}$. The metric $\hat d$ generates the topology of the space $\hat X$, $\rng(\pi)\subset\hat X$ is dense and $d=\hat d\circ\pi$.
The completeness of the interpreted system implies that $\hat d$ is complete. It follows that $\hat X$ is (isomorphic to) the completion of $\langle X,d\rangle$.
\end{proof}

\begin{example}
The Baire space or the Hilbert cube in a transitive model $M$ are interpreted as the Baire space or the Hilbert cube.
\end{example}

\begin{example}
The transitivity of the model $M$ is necessary in the assumptions of Corollary~\ref{metrizablecorollary}. Let $M$ be a model of set theory containing an $M$-ordinal $\ga$ which is illfounded. In the model $M$, consider the space $\ga^\gw$ with the usual minimum difference metric. The completion of $(\ga^\gw)^M$
is $\ga^\gw$, which contains an infinite decreasing sequence $x$. Now, by a wellfoundedness argument inside $M$, $M\models\ga^\gw=\bigcup\{O_t\colon t$ is a nondecreasing finite partial map from $\gw$ to $\ga\}$ where $O_t=\{y\in\ga^\gw\colon t\subset y\}$.
It is clear that $x$ does not belong to the natural interpretation of any of the sets $O_t$ in the union, and therefore $\ga^\gw$ is not the interpretation of $(\ga^\gw)^M$.
\end{example}

\begin{corollary}
\label{uniformcorollary}
Every complete uniform space with a countable complete uniform sequence of covers in a transitive model $M$ is interpreted as a complete uniform space.
\end{corollary}

\begin{example}
The countability cannot be removed from the assumptions of Corollary~\ref{uniformcorollary}. In a transitive model $M$, consider the space $X=\gw^{\gw_1}$ with the (uncountable) complete uniform collection of covers $\{C_a\colon a\subset\gw_1$ is finite$\}$,
where $C_a=\{O_t\colon t\in \gw^a\}$ and $O_t=\{x\in X\colon t\subset x\}$. If the interpretation preserved the completeness of this system of covers, it would have to be (equivalent to) the identity map from $X=(\gw^{\gw_1})^M$ to
$(\gw^{\gw_1})^V$. However, Example~\ref{productexample} describes a situation where the interpretation of $X$ is different.
\end{example}

\begin{corollary}
\label{cechcorollary}
Every {\v C}ech complete space in a transitive model $M$ is interpreted as a {\v C}ech complete space.
\end{corollary}

\begin{corollary}
\label{interpretablecorollary}
Every interpretable space in a transitive model $M$ is interpreted as an interpretable space.
\end{corollary}

To conclude this section, I provide a testable criterion for a map $\pi$ to be an interpretation; this will be used in several situations later.
Suppose that $M$ is a model of set theory and $M\models\langle X, \tau\rangle$ is a regular Hausdorff space. Suppose that $\pi\colon X\to\hat X, \tau\to\hat\tau$ is a topological preinterpretation. Then for every open set $O\in\tau$, $\pi(O)=\hat X\setminus {\mathrm{cl}} (\pi''(X\setminus O))$. The left-to-right inclusion is clear since $\pi(O)\subset\hat X$ is an open set disjoint from $\pi''(X\setminus O)$. For the right-to left inclusion use the fact that the range $\pi''\tau$ generates the topology $\hat\tau$: whenever $x\in \hat X\setminus {\mathrm{cl}} (\pi''(X\setminus O))$ is a point, it has to have an open neighborhood disjoint from $\pi''(X\setminus O)$, this neighborhood can be taken of the form $\pi(P)$ for some $P\in\tau$, and
since $\pi(P)$ contains no points in $\pi''(X\setminus O)$ it must be the case that $P\subset O$ and so $x\in \pi(P)\subset\pi(O)$
as desired.

The previous paragraph shows that a given map $\pi\colon X\to \hat X$ can be completed into a topological preinterpretation in at most one way. This justifies the following definition:

\begin{definition}
\label{extensiondefinition}
Suppose that $M$ is a model of set theory, $M\models \langle X, \tau\rangle$ is a regular Hausdorff space, $\langle\hat  X, \hat\tau\rangle$ is a topological space, and $\pi\colon X\to\hat X$ is a function. The \emph{canonical extension} of $\pi$ is the map $\bar \pi\colon\tau\to\hat\tau$ defined by $\bar \pi(O)=\hat X\setminus{\mathrm{cl}}(\pi''(X\setminus O))$.
\end{definition}

If the original map $\pi$ from $X$ to $\hat X$ is arbitrary, then the canonical extension $\bar\pi$ is easily seen to preserve inclusion and finite intersections, but on other accounts it may be very poorly behaved. Still, it is the only candidate for extending $\pi$ into a topological preinterpretation. 

\begin{proposition}
\label{testproposition}
Suppose that $M$ is a transitive model of set theory, $M\models \langle X, \tau\rangle$ is an interpretable space,
and $\pi\colon X\to\hat X$ is a map to a regular Hausdorff space $\langle \hat X, \hat \tau\rangle$.
Let $\bar \pi\colon \tau\to\hat\tau$ be the canonical extension of $\pi$.
Suppose that there is in the model $M$ a basis $\gs\subset\tau$ closed under intersections, and a complete sieve $\langle S, O(s)\colon s\in S\rangle$ such that

\begin{enumerate}
\item for every $x\in O\in\gs$ it is the case that $\pi(x)\in\bar\pi(O)$;
\item whenever $Q\in\gs$ and $\{P_j\colon j\in J\}$ is a finite subset of $\gs$ and $Q\subset \bigcup_jP_j$ then $\bar\pi(Q)\subset\bigcup_j\bar\pi(P_j)$;
\item $\bar\pi''\gs$ generates the topology $\tau$;
\item the sets on the sieve are in the basis $\tau$ and moreover $\langle S, \bar\pi(O(s))\colon s\in S\rangle$ is a complete sieve for $\hat X$.
\end{enumerate}

\noindent Then the map $\bar\pi\colon\tau\to\hat \tau$ is a topological interpretation.
\end{proposition}

\noindent Note that for a compact space $X$ the last item reduces to the demand that $\hat X$ is compact.

\begin{proof}
By Theorem~\ref{sievetheorem} and (4) it is only necessary to verify that $\bar\pi$ is a topological preinterpretation. It is clear from the definitions that the canonical extension $\bar\pi$ on $\tau$ preserves inclusion and commutes with finite intersections. The first item implies then that for every $x\in X$ and $O\in\tau$, it is the case that $x\in O\liff \pi(x)\in\bar\pi(O)$. The only thing left to verify is that the map $\bar\pi$ commutes with arbitrary unions of open sets in the model $M$.

Suppose then that $M\models O=\bigcup_{i\in I}O_i$ is a union of open sets and argue that $\bar\pi(O)=\bigcup_i\bar\pi(O_i)$ must hold.
The right-to-left inclusion follows from the fact that $\bar\pi$ preserves inclusion. For the left-to-right inclusion, assume for contradiction that the set $\bar\pi(O)\setminus\bigcup_i\bar\pi(O_i)$ is nonempty, containing some point $x\in\hat X$. 
By the third item, there is an open set $Q\in\gs$ such that $x\in\bar\pi(Q)$ and the closure of $\bar\pi(Q)$
is a subset of $\bar\pi(O)$. It is easy to argue from (1) and (3) that the closure of $Q$ must be a subset of $O$ in the space $X$. In the model $M$, let $\{P_j\colon j\in J\}\subset\gs$ be a collection such that $X\setminus{\mathrm{cl}}(Q)=\bigcup_jP_j$.

Let $b$ be any infinite branch of the tree $S$ such that $x\in \bigcap_n\bar\pi(O(b(n)))$. Then, no set $\bar\pi(O(b(n))$ is covered by finitely many sets in the collection $\{\bar\pi(O_i)\colon i\in I, \pi(P_j)\colon j\in J\}$. By (2) and a wellfoundedness argument
with the model $M$, in the model $M$ there must be an infinite branch $c$ such that no set $O(c(n))$  for $n\in\gw$ is covered by finitely many sets in the collection $\{O_i\colon i\in I, P_j\colon j\in J\}$. By the completeness of the sieve in the model $M$,
the set $X\setminus (\bigcup_iO_i\cup\bigcup_jP_j)$ must be nonempty. However, any element of this set belongs to $O$ and not to any $O_i$ for $i\in I$, contradicting the initial assumption.
\end{proof}

\section{Interpretations of Borel sets}

The purpose of this section is to show that for interpretable spaces,  it is possible to extend interpretations to the $\gs$-algebra of Borel sets so that the interpretation commutes with the algebraic operations.

\begin{theorem}
\label{boreltheorem}
Suppose that $M$ is a transitive model of set theory and $M\models\langle X, \tau\rangle$ is an interpretable space with its Borel $\gs$-algebra $\mathcal{B}$. Suppose moreover that $\pi\colon \langle X, \tau\rangle\to\langle \hat X, \hat \tau\rangle$ is an interpretation
and write $\hat{\mathcal{B}}$ for the Borel $\gs$-algebra of the space $\hat X$. Then there is a unique map $\pi\colon \mathcal{B}\to\hat{\mathcal{B}}$ such that it agrees with the action of $\pi$ on $\tau$, and commutes with complements and countable unions and intersections in $M$.
\end{theorem}

\begin{proof}
The uniqueness of the extension is clear as Borel sets are obtained from closed and open sets by complements, countable unions and intersections. The existence of the map is a much more complicated matter. The problem is that a Borel subset of $X$ in the model $M$ can be obtained from open and closed sets by countable unions and intersections in two different ways, and then the two different ways may yield different results when reinterpreted over the space $\hat X$. This does not happen due to the wellfoundedness
of the model $M$; this is the contents of the present proof.

Work in $M$. Let $S$ be a complete sieve for the space $X$. Select symbols $\cup$ and $\cap$ for union and intersection. Define a \emph{Borel code} by $\in$ recursion in the following way. If $A\subset X$ is a closed or open set, then $\{A\}$ is a Borel code and if $B$ is a countable set of Borel codes, then $\{\cup, B\}$, $\{\cap, B\}$ are Borel codes. By induction on the rank of the code $c$ define a set $B_c\subset X$ in the following way. If $c=\{A\}$ for an open or closed set $A\subset X$, then $B_c=A$; and if $c=\{\cup, D\}$, resp.\ $c=\{\cap, D\}$ then $B_c=\bigcup_{d\in D} B_d$, resp. $B_c=\bigcap_{d\in D} B_d$.

Still working in $M$, for every Borel code $c$ associate a certain ordering $T_c$. Let $\mathcal{P}$ be the set of all sequences of the form $s=\langle s_i, E_i\colon i\in n\rangle$ where $s_i$' s form a strictly descending sequence of
nodes in $S$, $E_i$'s form a strictly descending sequence of closed nonempty sets, and $E_i\subset O(s_i)$; write $E(s)=E_{n-1}$.
The set $\mathcal{P}$ is naturally ordered by extension. By $\in$-recursion on the Borel code $c$, define
orderings $T_c$. Elements of $T_c$ will always be finite tuples whose first coordinate is a play in $\mathcal{P}$, and for each such tuple $p=\langle s, y_0, y_1\dots\rangle$ I will write $E(p)=E(s)$.

\begin{itemize}
\item if $c=\{A\}$ for an open or closed set $A\subset X$, then $T_c$ consists of all pairs $\langle s, 0\rangle$ where $s\in\mathcal{P}$ and $E(s)\subset A$. The ordering is that of strict extension in the first coordinate;
\item if $c=\{\cup, D\}$ for some countable set $D$ of codes, then $T_c$ consists of all pairs $\langle s, d, u\rangle\rangle$ such that $s\in \mathcal{P}$, $d\in D$, $u\in T_d$ and $E(s)\subset E(u)$. The ordering is defined by $\langle t, d, u\rangle>\langle s, e, v\rangle$ if $t$ properly extends $s$, $e=d$, and $u>v$ in $T_d$;
\item  if $c=\{\cap, D\}$ for some countable set $D=\{d_i\colon i\in\gw\}$ of codes with a fixed enumeration, then $T_c$ consists of all tuples $\langle s, u_i\colon i\in n\rangle$ where $s\in\mathcal{P}$ has length $n$, for every $i\in n$, $u_i\in T_{d_i}$, and $E(s)\subset E(u_i)$.
The ordering is defined by $\langle s, u_i\colon i\in n\rangle> \langle t, v_i\colon i\in m\rangle$ if $t$ properly extends $s$ and for each $i\in n$, $u_i>v_i$ in the ordering $T_{d_i}$.
\end{itemize}

\begin{claim}
\label{borelclaim1}
In the model $M$: If $p$ is an infinite descending sequence in $T_c$ then $\bigcap_nE(p(n))$ is a nonempty set which is a subset of $B_c$.
\end{claim}

\begin{proof}
This is proved by an elementary $\in$-induction argument on the code $c$.
\end{proof}

Now move out of the model $M$. Let $\pi\colon X\to\hat X$ be an interpretation. For every closed set $C\subset X$ in the model $M$, write $\pi(C)={\mathrm{cl}}(\pi''C)$ which is equal to $\hat X\setminus\pi(X\setminus C)$. For every Borel code $c\in M$ define a set $\hat B_c\subset\hat X$ in the following way.  If $c=\{A\}$ for a closed or open set $A\subset X$, then $\hat B_c=\pi(A)$;  and if $c=\{\cup, D\}$, resp.$c=\{\cap, D\}$ then $\hat B_c=\bigcup_{d\in D}\hat B_d$, resp. $\hat B_c=\bigcap_{d\in D}\hat B_d$.

\begin{claim}
\label{borelclaim2}
For every Borel code $c\in M$ and every element $x\in\hat B_c$ there is an infinite descending sequence $p$ in $T_c$ such that $x\in\bigcap_n\pi(E(p(n)))$. 
\end{claim}

\begin{proof}
By induction on the rank of the code $c$. The only interesting case is that of countable intersection, so $c=\langle \bigcap, D\rangle$
for some countable set $D=\{d_i\colon i\in\gw\}$ of codes with a fixed enumeration in the model $M$.

Suppose that $c=\{\cap, D\}$ and $x\in \hat B_c$. By the induction hypothesis, for every $i\in \gw$ there is an infinite descending sequence $p_i$ in $T_{d_i}$ such that $x\in\bigcap_n\pi(E(p_i(n)))$. By induction on $n\in\gw$ build sequences $t_n\in \mathcal{P}$ of length $n$ such that $t_0\subset p_1\subset\dots$, $x\in E({t_n})$, and $\langle t_n, p_i(n)\colon i\in n\rangle\in T_c$. Once this is done, the sequence $p$ given by $p_n=\langle t_n, p_i(n)\colon i\in n\rangle$ is as required.

To perform the induction step, suppose that the sequence $t_n$ has been found, with last pair $s\in S$ and $E\subset X$ closed. In the model $M$, $A=\{X\setminus E\}\cup \{O(t)\colon t\in S$ is a one step extension of $s$ in $S\}$ is an open cover of the space $X$. 
As $\pi$ is a preinterpretation,
$\pi''A$ is an open cover of $X$ and so there must be a one step extension $s'\in S$ of $s$ such that $x\in\pi(O)$. Let $t_{n+1}=t_n^\smallfrown s', \bigcap_{i\in n+1}E(p_i(n+1))\cap{\mathrm{cl}}(O)$. The induction step has been performed.
\end{proof}

\begin{claim}
\label{borelclaim3}
Suppose that $c\in M$ is a Borel code. Then $B_c=0$ if and only if $\hat B_c=0$.
\end{claim}

\begin{proof}
The right-to-left implication is easier. By elementary $\in$-induction on the code $c$ show that for every $x\in X$, $x\in B_c\liff x\in\hat B_c$.
Thus, if $B_c\neq 0$ and $x\in B_c$ then $\hat B_c\neq 0$ as well, since $\pi(x)\in\hat B_c$.

The left-to-right implication is harder, and it uses the wellfoundedness of the model $M$. If $\hat B_c\neq 0$ then by Claim~\ref{borelclaim2}
there is an infinite descending sequence in the ordering $T_c$. Since the model $M$ is wellfounded, there must be such an
infinite descending sequence in the model $M$ as well. By Claim~\ref{borelclaim1}, $B_c\neq 0$ as desired.
\end{proof}

\begin{claim}
\label{borelclaim4}
Suppose that $c, d\in M$ are Borel codes. Then $B_c\subseteq B_d$ if and only if $\hat B_c\subseteq \hat B_d$.
\end{claim}

\begin{proof}
First, work in the model $M$. By $\in$-recursion on the code $c$ define a code $\lnot c$.
If $c=\{A\}$ for some open or closed set $A\subset X$ then let $\lnot c=\{ X\setminus A\}$. If $c=\{\cup, D\}$ then
$\lnot c=\{\cap, \{\lnot d:d\in D\}\}$ and  if $c=\{\cup, D\}$ then
$\lnot c=\{\cup, \{\lnot d:d\in D\}\}$. It is can be proved by an immediate $\in$-induction on the code $c$ that
$B_c=X\setminus B_{\lnot c}$ and $\hat B_c=X\setminus\hat B_{\lnot c}$.

Now suppose that $c,d\in M$ are Borel codes. Consider the code $e=\{\cap, \{c, \lnot d\}$. Then $B_c\subseteq B_d=0$ if and only if $B_e=0$ if and only if (Claim~\ref{borelclaim3}) $\hat B_e=0$ if and only if $\hat B_c\subset\hat B_d$ as desired.
\end{proof}

Finally, we are ready to extend the preinterpretation $\pi$ to Borel sets. Let $B\in M$ is a Borel set; define $\pi(B)=\hat B_c$
for any Borel code $c\in M$ such that $B=B_c$. Claim~\ref{borelclaim4} shows that this definition does not depend on the choice of the code $c$.
Now, suppose that $M\models B=\bigcup_nB_n$ is a countable union of Borel sets. There must be a sequence
$\langle c_n:n\in\gw\rangle\in M$ of Borel codes such that $B_n=B_{c_n}$ for every $n\in\gw$. Consider the Borel code
$d=\{\cup, \{c_n:n\in\gw\}\}$. Then clearly $B=B_d$, and $\pi(B)=\hat B_d=\bigcup_n\hat B_{c_n}=\bigcup_n\pi(B_n)$ as desired.
The countable intersection is handled in the same way.
\end{proof}

As a result, whenever I encounter an interpretation $\pi\colon X\to\hat X$ of an interpretable space and a Borel set $B\subset X$, I will freely use the symbol $\pi(B)$ to refer to the Borel subset of $\hat X$ which is the image of $B$ under the
unique extension of $\pi$ to the $\gs$-algebra of Borel sets.

\begin{example}
The conclusion of the theorem may fail for the most common non-interpretable spaces. Consider the space $X=\mathbb{Q}$ with the usual topology and a generic extension
$V[G]$ containing a Cohen real $r\in\mathbb{R}$. Let $\pi\colon X\to\hat X$ be an interpretation; it cannot be extended to an interpretation of Borel sets.
To see this, note that the extension would have to assign $\pi\{x\}=\{\pi(x)\}$ for every singleton $x\in X$, and since $X$ is a countable union of singletons, $\pi(X)$ would have to be equal to $\pi''X$.
This contradicts the conclusion of Example~\ref{isolatedexample}.
\end{example}

\begin{example}
The conclusion holds in general for some non-interpretable spaces. If $\langle X, \tau\rangle$ is an interpretable space and $\gs$ is an alternative topology on it which extends $\tau$, consists of only
$\tau$-Borel sets and is Lindel{\" o}f, then the interpretation of $\langle X, \tau\rangle$ naturally extends to an interpretation of $\langle X, \gs\rangle$. Since every $\gs$-Borel set is also $\tau$-Borel, the theorem
provides for an extension of the interpretation to the algebra of $\gs$-Borel algebra. At the same time, the space $\langle X, \gs\rangle$ may not be interpretable. A good example of this behavior is the Sorgenfrey line as an extension
of the usual Euclidean topology on the real line. Sorgenfrey line is not interpretable by Example~\ref{sorgenfreyexample}. 
\end{example}

\begin{example}
The conclusion of the theorem may fail for illfounded models $M$, even if their $\gw_1^M$ is wellfounded.
Suppose that $M$ is a model of set theory and $\ga$ its ordinal which is illfounded in $V$ in it. In $M$, consider the space $\ga+1$ with order topology  and the space $X=(\ga+1)^\gw$ in the model $M$. By Theorem~\ref{compactproducttheorem} (which does not need the assumption of wellfoundedness of the model $M$), the interpretation of $X$ is the natural map to the space $\hat X=(\hat\ga+1)^\gw$, where $\hat\ga$ is the completion of the linear ordering on $\ga$ in $V$. For every $n\in\gw$, let $O_n=\{x\in X\colon x(n)>x(n+1)\}$; this is an open set in the model $M$. Since $M\models\ga$ is wellfounded, it is the case that $M\models\bigcap_nO_n=0$. On the other hand, any infinite descending sequence in $\ga$ is an element of $\bigcap_n\pi(O_n)$.
\end{example}

\begin{corollary}
Suppose that $M$ is a transitive model of set theory, $M\models \langle X, \tau\rangle$ is an interpretable space and $\langle C_n\colon n\in\gw\rangle$ is a development on $X$. Let $\pi\colon X\to \hat X$ be an interpretation. Then
$\langle \pi''C_n\colon n\in\gw\rangle$ is a development on $\hat X$. In particular, interpretable Moore spaces are interpreted as Moore spaces.
\end{corollary}

\begin{proof}
In the model $M$, consider an open set $O\in\tau$. Let $P_n(O)=\bigcup\{Q\in C_n\colon Q\not\subset O\}$. Since $\langle C_n\colon n\in\gw\rangle$ is a development, $O\cap\bigcap_nP_n=0$. 
Now, step out of the model $M$. By Theorem~\ref{boreltheorem}, $\pi(O)\cap\bigcap_n\pi(P_n(O))=0$ holds for every open set $O\in\tau$, which is to say that for every point $x\in\pi(O)$
there is a number $n\in\gw$ such that $x$ belongs to no set in the cover $\pi''C_n$ which has nonempty intersection with the complement of $O$. Since every open set in the space $\hat X$ is a union of open sets in the range of the interpretation $\pi$,
it follows that $\langle\pi''C_n\colon n\in\gw\rangle$ is a development on $\hat X$ as desired.
\end{proof}

\begin{corollary}
\label{subspacecorollary}
Suppose that $M$ is a transitive model of set theory and $M\models\langle X, \tau\rangle$ is an interpretable space and $Y\subset X$ its $G_\gd$ subset. Suppose that $\pi\colon X\to\hat X$ is an interpretation. Then the function $\chi=\pi\restriction Y$, $\chi\colon Y\to\pi(Y)$
extends to an interpretation of the space $Y$.
\end{corollary}

\begin{proof}
For every open set $O\in\tau$, let $\chi(O\cap Y)=\pi(O)\cap\pi(Y)$. This depends only on $O\cap Y$ and not on all of $O$ by Theorem~\ref{boreltheorem}. The fact that $\pi$ is an interpretation immediately implies that $\chi$ is a preinterpretation.
To show that $\chi$ is in fact an interpretation, I will produce a complete sieve on $Y$ whose $\chi$-image remains complete on $\pi(Y)$ and then use Theorem~\ref{sievetheorem}.

Work in the model $M$. Let $Y=\bigcap_nQ_n$ be a countable intersection of an inclusion-decreasing sequence of open sets. It is easy to adjust an arbitrary strong complete sieve $\langle S, O(s)\colon s\in S\rangle$ on the space $X$ to one with the following property:
(*) if $s\in S$ has length $n$ then $O(s)\cap Q_n=\bigcup\{O(t)\colon t$ is an immediate successor of $s$ and ${\mathrm{cl}}(O(t))\subset Q_n\}$. Now let $T=\{s\in S$: for every $n\leq\dom(s)$ greater than $0$, $O(t\restriction n)\subset Q_{n-1}\}$.
It is clear that $T$ is a tree. Let $\langle T, P(t)\colon t\in T\rangle$ be defined by $P(t)=Y\cap O(t)$. It is immediate from (*) that $\langle T, P(t)\colon t\in T\rangle$ is a sieve on the space $Y$. It is also complete by a repeated use of Fact~\ref{sieveproposition}: for every finitely branching tree
$U\subset T$, the set $\bigcap_n\bigcup\{{\mathrm{cl}}(O(t))\colon t\in U, |t|=n\}\subset X$ is compact, by (*) it is a subset of $Y$, and so it is equal to  $\bigcap_n\bigcup\{{\mathrm{cl}}_Y(P(t))\colon t\in U, |t|=n\}\subset Y$ which must then be compact.

Now, move out of the model $M$. The reasoning of the whole previous paragraph is transported by $\pi$ without damage. The only notable point is that the sieve $\langle S, \pi(O(s))\colon s\in S\rangle$ is complete on the space $\hat X$
by Theorem~\ref{sievetheorem}. It follows that the sieve $\langle T, \chi(P(t))\colon t\in T\rangle$ is complete on the space $\pi(Y)$. By Theorem~\ref{sievetheorem} again, the map $\chi$ is an interpretation as desired.
\end{proof}

\section{Products}

In order to speak about relations and fnctions on topological spaces, it is necessary to evaluate interpretations of products. The two theorems presented in this section both rely on a basic feature of finite products:

\begin{proposition}
\label{productproposition}
Suppose that $\{Y_j\colon j\in J\}$ is a finite collection of topological spaces, and $\prod_jY_j$ is covered by a finite collection
of rectangular boxes $\{B_k\colon k\in K\}$. Then there are finite open covers $\{C_j\colon j\in J\}$ of the respective spaces $Y_j$
such that whenever $O_j\in C_j$ for $j\in J$ is a selection, then there is $k\in K$ such that $\prod_jO_j\subset B_k$.
\end{proposition}

\begin{proof}
For every $j\in J$ and every $x\in X_j$ let $O_x$ be the intersection of all open subsets of $X_j$ which serve as a side of some of the boxes $B_k$ and contain the point $x$. Let $C_j=\{O_x\colon x\in X_j\}$. It is easy to verify that these covers work.
\end{proof}

\begin{theorem}
\label{compactproducttheorem}
Suppose that $M$ is a model of set theory and $M\models X_i$ are compact Hausdorff spaces for $i\in I$ and write $X=\prod_iX_i$. Suppose that for every $i\in I$, $\pi_i\colon X _i\to\hat X_i$ is an interpretation and write $\hat X=\prod_i\hat X_i$. Then 

\begin{enumerate}
\item the product map $\pi=\prod_i\pi_i\colon X\to\hat X$ extends to an interpretation of the product;
\item the coordinate projection functions are interpreted as the coordinate projection functions.
\end{enumerate}
\end{theorem}

\begin{proof}
Define a \emph{basic function} on $X$ to be a function $g$ on a finite set $J\subset I$ such that for every $j\in J$, $g(j)\in\tau_j$. If $g$ is a basic function then let $O(g)=\{x\in X\colon \forall j\in J x(j)\in g(j)\}$. Let $\gs=\{O(g)\colon g$ is a basic function on $X\}$; this is a basis for the space $X$. Also, define $O(\pi g)=\{x\in\hat X\colon \forall j\in J\ x(j)\in \pi(g(j))\}\subset\hat X$. The set $\hat \gs=\{O(\pi g)\colon g$ is a basic function on $X\}$ is a basis for $\hat X$. The following claim records the relationship between $\gs$ and $\hat\gs$.

\begin{claim}
Let $g$ be a basic function on $X$.

\begin{enumerate}
\item For every $x\in X$, $x\in O(g)\liff\pi(x)\in O(\pi g)$;
\item if $h_k\colon k\in K$ is a finite set of basic functions and $O(g)\subset\bigcup_kO(h_k)$
then $O(\pi g)\subset \bigcup_kO(\pi h_k)$.
\end{enumerate}
\end{claim}

The first item is proved by unraveling the definitions, and the second follows from Proposition~\ref{productproposition} applied in the model $M$. Now, let $\hat\pi$ be the canonical extension of the product map $\pi$ to the topology of $X$ as in Definition~\ref{extensiondefinition}. Note that $\hat(O(g))=O(\pi(g))$ holds for every basic function $g$ on $X$: for the right-to-left inclusion, observe that $O(\pi g)$ is an open set disjoint from $\pi''(X\setminus O(g))$ by (1) of the claim. For the left-to-right inclusion, if $h$ is a basic function such that $O(\pi h)$ is disjoint from $\pi''(X\setminus O(g))$ then $O(h)\subset O(g)$
by (1) of the claim, and $O(\pi h)\subset O(\pi g)$ by (2) of the claim. 

Finally, apply Proposition~\ref{testproposition} to see that $\hat\pi$ is an interpretation, using the fact that the space $\hat X$ is compact. (2) of the theorem is then immediate.
\end{proof}

\begin{theorem}
\label{countableproducttheorem}
Let $M$ be a transitive model of set theory and $M\models\langle X_i, \tau_i\rangle$ for $i\in\gw$ are interpretable spaces. Let $\pi_i\colon X _i\to\hat X_i$ be interpretations for every $i\in\gw$.

\begin{enumerate}
\item  The product map $\pi=\prod_i\pi_i\colon X=\prod_iX_i\to\hat X =\prod_i\hat X_i$ can be extended to an interpretation of the product space;
\item the coordinate projection functions are interpreted as the coordinate projection functions.
\end{enumerate}
\end{theorem}

\begin{proof}
The proof follows the argument for Theorem~\ref{compactproducttheorem} letter by letter except for the last paragraph.
To replace the compactness argument in the last paragraph, work in $M$. For every $i\in\gw$ find complete sieves
$\langle S_i, O_i(s)\colon s\in S_i\rangle$ on each space $X_i$ in the model $M$. Let $\langle T, P(t)\colon t\in T\rangle$ be the product sieve on the space $X$ as desribed in the proof of Theorem~\ref{interpretabletheorem}; it is complete and its open sets
are rectangular boxes with finite support.
Step out of the model $M$.
The sieves $\langle S_i, \pi_i(O_i(s))\colon s\in S_i\rangle$ are complete sieves for each space $\hat X_i$ by Theorem~\ref{sievetheorem}. Their product sieve is again complete on the space $\hat X$. Thus, the canonical extension $\bar\pi$ of $\pi$ maps a complete sieve to a complete sieve, and Proposition~\ref{testproposition} is applicable again to show that $\bar\pi$ is an interpretation.
\end{proof}

\begin{corollary}
Let $M$ be a transitive model of set theory and $M\models\langle X, \tau\rangle, \langle Y, \gs\rangle$
are interpretable spaces and $f\colon X_0\to X_1$ is a continuous function. 
Then $\hat f$ is just the interpretation of $f$ viewed as a closed subset of $X\times Y$.
\end{corollary}

\begin{proof}
Let $\pi\colon X\to\hat X$ and $\chi\colon Y\to\hat Y$ be interpretations. Then $\hat f\subset\hat X\times \hat Y$ is a closed set containing the set $g=\{\langle \pi(x), \chi(f(x)\rangle\colon x\in X\}$.
The set $(\pi\times\chi)(f)\subset X\times Y$ is exactly the closure of the set $g$. Thus, it will be enough to show that every vertical section of the set $(\pi\times\chi)(f)$ is nonempty. One illuminating way to see this
is to note that the projection function from $X\times Y$ to $X$, when restricted to the graph of $f$, is a homeomorphism. Thus, it must be interpreted as a homeomorphism of $(\pi\times\chi)(f)$ and $X$ by Corollary~\ref{homeocorollary}.
\end{proof}

\begin{corollary}
\label{disjointrangecorollary}
Suppose that $M$ is a transitive model of set theory and $M\models X, Y, Z$ are interpretable spaces and $f\colon X\to Z$ and $g\colon Y\to Z$ are continuous functions.
Let $\pi\colon X\to\hat X$, $\chi\colon Y\to\hat Y$, $\xi\colon Z\to\hat Z$, and $\hat f\colon \hat X\to \hat Z$ and $\hat g\colon\hat Y\to\hat Z$ be interpretations.
If $\rng(f)\cap\rng(g)=0$ then $\rng(\hat f)\cap\rng(\hat g)=0$.
\end{corollary}

\begin{proof}
In the model $X$, consider the product $X\times Y\times Z$ and use the fact that $\rng(f)\cap\rng(g)=0$ to see that it is the union of sets $O\times P\times Q$ such that either $f^{-1}Q\cap O=0$ or $g^{-1}Q\cap P=0$. By the theorem, the product
of interpretations is again an interpretation and thus the union of the corresponding open sets covers the whole product. This means that $\rng(\hat f)\cap\rng(\hat g)=0$ by Theorem~\ref{functiontheorem}(2). 
\end{proof}

\begin{corollary}
\label{injectioncorollary}
If $M\models f\colon X\to Y$ is a continuous injection between two interpretable spaces, then $\hat f$ is injection again.
\end{corollary}

\begin{proof}
Immediate from Corollary~\ref{disjointrangecorollary}.
\end{proof}

I conclude this section with instructive examples of pathological behavior in products for spaces that do not fall into the interpretable category.

\begin{example}
The product of interpretations the space of rational numbers in a Cohen forcing extension does not extend to a preinterpretation of the product. To see this,
Let $X, Y\subset\mathbb{Q}$ be two disjoint dense sets of rationals with the inherited topology. Thus, both are homeomorphic to the rationals. The collection $C=\{\langle I\times J\rangle\colon I, J$ are disjoint open intervals of rational numbers$\}$ is a cover of the space $X\times Y$.

Now, let $r\in\mathbb{R}$ be a Cohen real and work in $V[r]$. Suppose that $\pi\colon X\to\hat X$ and $\chi\colon Y\to\hat Y$ are interpretations. The computation of the interpretation of the space of rational numbers
shows that there are elements $x\in \hat X$ and $y\in\hat Y$ such that for every open interval of rational numbers with rational endpoints, $x\in \pi(I)$ if and only if $r$ is between the endpoints of $I$, and similarly for $y$. Clearly, the point $\langle x, y\rangle\in\hat X\times\hat Y$
does not belong to the union of interpretations of the rectangles in the cover $C$.
\end{example}

\noindent The problem in the previous example is apparently caused by the fact that the interpretation of $\mathbb{Q}$ does not respect the Borel structure on the space. The difficulty disappears if one considers only interpretations of topological spaces with Borel structure
as in Section~\ref{borelsection}.
The next less trivial counterexample will work even then:

\begin{example}
Let $X$ be the space of all wellfounded trees on $\gw$, with the topology inherited from $\power(\gw\times\gw)$. In every generic extension collapsing $\mathfrak{c}$ to $\aleph_0$, the product of interpretations of $X$ and $\baire$ may not extend to an interpretation of $X\times\baire$.
 The set of all wellfounded trees is viewed as a subspace of the space $\power(\gwtree)$ with its usual Polish topology.
Let $V[G]$ be some generic extension collapsing $\mathfrak{c}$
and work in the model $V[G]$. It will be enough to find a Borel topological preinterpretation $\pi\colon X \to\hat X$ such that $\hat X$ contains an illfounded tree $T\subset\gwtree$. Consider the open sets $O_t=\{\langle S, y\rangle\in X\times\baire\colon t\notin S\land t\subset\baire\}\subset X\times\baire$ for $t\in\gwtree$. It is immediate that the sets $O_t\subset X\times\baire$ are open rectangles and $\bigcup_tO_t=X\times\baire$ holds in $V$. However, if $y\in\baire$ is a branch through the illfounded tree $T$, the pair $\langle T, y\rangle$ is not covered by any of the interpreted rectangles. 

To find the space $\hat X$, consider the union $B$ of all interpretations of ground model Borel subsets of $\power(\gwtree)$ containing only illfounded trees. By the Shoenfield absoluteness, the interpretations also contain only illfounded trees, and as there are only countably many ground model Borel sets, the set $B$ is Borel and contains only illfounded trees. The
set of wellfounded trees is not Borel, so there must be an illfounded tree $T\notin B$. Let $\hat X=X\cup\{T\}$, let $\pi\colon X\to\hat X$ be the identity map,
and for every Borel set $B\subset\power(\gwtree)$ in the ground model, let $\pi(C\cap X)$ be the intersection of the interpretation of $C$ with $\hat X$. It is important to observe that this definition depends only on $C\cap X$ by the choice of the tree $T$. In the ground model, if $C, D\subset\power(\bintree)$ are two Borel sets in the ground model such that $C\cap X=D\cap X$, then $C\Delta D$ is a Borel set consisting of illfounded trees only. By the choice of the tree $T$, $T$ belongs to the interpretation
of $C$ if and only if it belongs to the interpretation of $D$. It is immediate now to check that $\pi$ is a preinterpretation of the space $X$.
\end{example}

\begin{example}
\label{sorgenfreyexample}
Let $X$ be the Sorgenfrey line. In every generic extension adding a new real, the product of topological interpretations
of $X$ cannot be extended to a topological preinterpretation of $X\times X$. Write $\tau$ for the Sorgenfrey topology and $\gs$ for the Euclidean topology on $X$ and move to a generic extension $V[G]$ containing a new real $r\in \mathbb{R}$.

First, it is easy to evaluate the interpretation of $\langle X, \tau\rangle$. Note that every Sorgenfrey open set is an open set of reals together with countably many points, and moreover a union of a collection of Sorgenfrey open sets is equal to a union of a countable subcollection. Moreover, every open set of reals is Sorgenfrey open.
This means that every preinterpretation of $\gs$ can be uniquely extended to an preinterpretation of $\tau$,
and every preinterpretation of $\tau$ can be restricted to a preinterpretation of $\gs$. It is immediate to conclude
that the interpretation of $\tau$ is just the space $\langle\hat X, \hat\tau\rangle$ where
$\hat X$ is the set of all reals and $\hat\tau$ is the topology on $\hat X$ generated by closed-open intervals $[r,s)$ where $r, s$ are ground model reals, together with the obvious map $\pi$.

I will now show that the product map $\pi\times\pi$ cannot be even defined as a preinterpretation of the ground model Sorgenfrey plane.
In the model $M$, consider the collection $A$ of all open sets of the form $[r_0, s_0)\times [r_1, s_1)$ such that
either $-s_1<s_0$ or else $r_0=-r_1$. The union of the rectangles of the first kind covers the part of the plane below the negative diagonal; the union of the rectangles of the second kind covers the diagonal an the part of the plane above it.
Thus, $M\models\bigcup A=X\times X$. On the other hand, the union of the products of interpretations of the intervals
does not cover the plane $\mathbb{R}\times\mathbb{R}$: if $r\in\mathbb{R}$ is a real which is not in the ground model, then the point $\langle r, -r\rangle$ does not belong to the union.
\end{example}

\section{Interpretable structures}

Most topological spaces in practice come equipped with useful structures. These structures can be interpreted along with the spaces in question. It is natural to hope that the properties of the interpreted structures will not be far from the properties of the original structures. This section contains what I know about this problem at this point.

\begin{definition}
A \emph{interpretable structure} is a tuple $\mathfrak{X}=\langle X_i\colon i\in I, R_j\colon j\in J, f_k\colon k\in K\rangle$ where $X_i$ are interpretable spaces--the constituent spaces of $\mathfrak{X}$, $R_j$ are Borel finitary relations between the various spaces and $f_k$ are finitary partial continuous functions with closed or $G_\gd$ domain and range in one of the spaces.
\end{definition}

There are many interpretable structures commonly considered in mathematics, including topological groups, group actions, normed vector spaces with their duals and the application functions and so on. They can be interpreted in an obvious way:

\begin{definition}
Suppose that $M$ is a transitive model of set theory and let $M\models\mathfrak{X}=\langle X_i\colon i\in I, R_j\colon j\in J, f_k\colon k\in K\rangle$ is an interpretable structure on a space $X$. An \emph{interpretation} of $\mathfrak{X}$, written for short as $\pi\colon \mathfrak{X}\to\mathfrak{\hat X}$ is a structure $\hat{\mathfrak{X}}=\langle\hat X_i\colon i\in I,\hat R_j\colon j\in J,\hat f_k\colon k\in K\rangle$ with the same signature as $\mathfrak{X}$, and a tuple of constituent interpretations $\pi_i\colon X_i\to\hat X_i$ such that for each $j\in J$, the relation $\hat R_j$ is interpreted as the image of $R_j$ under the appropriate product of the maps $\pi_i$, and  for each $k\in K$, the function $\hat f_k$ is interpreted as the image of $f_k$ under the appropriate product of the maps $\pi_i$.
\end{definition}

The general expectation is that the interpreted structure will be at least in some ways similar to the original structure in the model $M$.
The following is the best general theorem in this direction.

\begin{theorem}
\label{absolutenesstheorem}
\textnormal{(Analytic absoluteness)} Suppose that $M$ is a transitive model of set theory and let $M\models\mathfrak{X}$ is a {\v C}ech complete structure on a space $X$. Suppose that $\pi\colon X\to\hat X$ is an interpretation of the space $X$. The map $\pi$ is a $\gS_1$-elementary embedding of the structure $\mathfrak{X}$ to $\hat{\mathfrak{X}}$.
\end{theorem}

\begin{proof}
Assume for simplicity that $\mathfrak{X}$ has a single constituent space $X$.

First, by an elementary induction on complexity of a quantifier free formula $\psi(\vec x)$, show that, writing $n$ for arity of $\psi$,
the set $B_\psi=\{\vec x\in X^n\colon\mathfrak{X}\models\psi(\vec x)\}\subset X^n$ is a Borel set in the model $M$,
and the interpretation $(\psi^n)(B_\psi)$ is exactly the set $\{\vec x\in\hat X^n\colon\hat{\mathfrak{X}}\models\psi(\vec x)\}$.

Now, let $\phi$ be a $\gS_1$ formula, $\phi(\vec x)=\exists\vec y\ \psi(\vec x, \vec y)$ where $\psi$ is quantifier free.
Suppose $\vec x$ is a finite string of elements of $X$. If $\mathfrak{X}\models\phi(\vec x)$ then there is $\vec y$
such that $\mathfrak{X}\models\psi(\vec x, \vec y)$ holds. Then, by the first paragraph, $\hat{\mathfrak{X}}\models\psi(\pi(\vec x), \pi(\vec y))$ and so $\hat{\mathfrak{X}}\models\phi(\pi(\vec x))$ as desired. If, on the other hand, $\mathfrak{X}\models\lnot\phi(\vec x)$, then the Borel set $B=\{\vec y\colon\mathfrak{X}\models\psi(\vec x, \vec y)\}$ is empty,
by the first paragraph it is interpreted as  $\{\vec y\colon\hat{\mathfrak{X}}\models\psi(\pi(\vec x), \vec y)\}$, and at the same time it is interpreted as the empty set. Thus, $\mathfrak{X}\models\lnot\phi(\pi(\vec x)$ as desired and the proof is complete.
\end{proof}

\begin{example}
The interpretation of the ordered field $\mathbb{R}$ of a transitive model $M$ is the ordered field $\mathbb{R}$.
The axioms of ordered fields are $\Pi_1$. The real ordering is complete without endpoints, and the field is Archimedean--these two
features characterize the real numbers. They also survive the interpretation process, the former by Corollary~\ref{compact1corollary} and the latter by the fact that an interpretation commutes with unions of open sets, so $\mathbb{R}=\bigcup_n(-n, n)$ is preserved.
\end{example}

Theorem~\ref{absolutenesstheorem} shows that if a structure $\mathfrak{X}$ in the model $M$ belongs to a class which is axiomatizable with $\gS_1$ and $\Pi_1$ formulas, then its interpretation belongs to the same class. As an example, the interpretation of a topological group is a topological group, the interpretation of a continuous group action is a continuous group action, the interpretation of a compatible metric is a compatible metric, similarly for Banach spaces or Hilbert spaces etc.

A persistent problem in the interpretation theory is the following. Suppose that a $\Pi_1$ ($\gS_1$ etc.) formula defines a topologically simple set (closed, open, Borel etc.) One would like to conclude that the same formula defines the interpretation of the set in the interpreted structure. This is by no means an automatic matter. The following theorem offers an affirmative answer to the absoluteness question which is easy to apply in numerous cases.

\begin{definition}
Let $\mathfrak{X}$ be an interpretable structure with a constituent space $X$.
A $\Pi_1$ formula $\phi$ \emph{absolutely defines a closed set} if for every poset $P$, $P\Vdash$ whenever $\pi\colon\mathfrak{X}\to\hat{\mathfrak{X}}$ is an interpretation then the formula $\phi$ in the structure $\hat{\mathfrak{X}}$ defines a closed subset of $X$.
\end{definition}

\noindent While the definition may look awkward, in practice it is normally the case that there is a ZFC proof that the formula $\phi$ defines a closed set in all structures similar to $\mathfrak{X}$, and then $\phi$ absolutely defines a closed set in $\mathfrak{X}$.

\begin{theorem}
\label{shoenfieldabsoluteness}
\textnormal{(Shoenfield absoluteness)}
Suppose that $M$ is a transitive model of set theory containing all ordinals. Suppose $M\models\mathfrak{X}$ is an interpretable structure with a constituent space $X$ and $\phi$ is a $\Pi_1$ formula that absolutely defines a closed set. Let $\pi\colon \mathfrak{X}\to\hat{\mathfrak{X}}$ be an interpretation. Then 
$$\pi(\{x\in X\colon\mathfrak{X}\models\phi(x)\})=\{x\in\hat X\colon\hat{\mathfrak{X}}\models\phi(x)\}.$$
\end{theorem}

\begin{proof}
It is easy to see that the statement can be reduced to the following. Suppose that $M\models X,Y$ are interpretable spaces and $B\subset X\times Y$ is a Borel set such that the interpretation of $B$ in all forcing extensions projects into an open subset of the interpretation of $X$. Write $O=p(B)\subset X$, where $p$ is the projection from $X\times Y$ to $X$; thus $O\subset X$ is an open set. Let $\pi_0\colon X, Y\to\hat X, \hat Y$ be an interpretation. I need to conclude that $\pi_0(O)=p(\pi(B))$.

To prove this statement, in the model $M$ consider the L\' evy collapse poset $P$ collapsing the size of bases of $X$ and $Y$ to $\aleph_0$. Let $G\subset P$ be a generic filter over $V$. In the model $M[G]$, let $\chi_0\colon X, Y\to \hat X_0, \hat Y_0$ be interpretations over the model $M$.

\begin{claim}
$M[G]\models\chi_0(O)=p(\chi_0(B))$.
\end{claim}

\begin{proof}
Let $\dot \chi$ be a $P$-name for interpretations of $X, Y$ respectively. By the homogeneity of the poset $P$, the set $A=\{Q\in\tau\colon \exists p\in P\ p\Vdash\dot \chi(Q)\subset p(\dot \chi(B))\}$ is equal to $\{Q\in\tau\colon P\Vdash\dot \chi(Q)\subset p(\dot \chi(B))\}$ and therefore is in the model $M$. By the assumption on the Borel set $B$, the projection of $\dot \chi(B)$ is forced to be open and therefore equal to $\bigcup\dot\chi''A$. By the analytic absoluteness~\ref{absolutenesstheorem}, $O=\bigcup A$. Since $\dot \chi$ is forced to be an interpretation, $\chi_0(O)=\chi_0(\bigcup A)=\bigcup\chi_0''A=p(\chi_0(B))$ as desired.
\end{proof}

In the model $M[G]$, the underlying spaces $\hat X_0, \hat Y_0$ are second countable, interpretable, and therefore Polish. In the model $V[G]$, let $\chi_1\colon X_0, Y_0\to X_1, Y_1$ be an interpretation over the model $M[G]$. By a standard
Shoenfield absoluteness argument \cite[Theorem 25.20]{jech:set} between the models $M[G]$ and $V[G]$ and the claim, $\chi_1(\chi_0(O))=p(\chi_1(\chi_0(B)))$. Now, by faithfulness~\ref{faithfulnesstheorem} applied to the chain $M\subset M[G]\subset V[G]$ of models, $\chi_1\circ\chi_0\colon X, Y\to
X_1, Y_1$ is an interpretation over the model $M$. By faithfulness~\ref{faithfulnesstheorem} applied to the chain $M\subset V\subset V[G]$ of models, there must be an interpretation $\pi_1\colon \hat X, \hat Y\to X_1, Y_1$ over $V$ such that $\pi_1\circ pi_0=\chi_1\circ\chi_0$.
Let $x\in\hat X$ be an arbitrary point. By the analytic absoluteness~\ref{absolutenesstheorem} applied to the interpretation $\pi_1$, if $x\in\pi_0(O)$ then $\pi_0(B)_x\neq 0$ and if $x\notin\pi_0(O)$ then $\pi_0(B)_x=0$. Thus, $\pi_0(O)=p(\pi_0(B))$ as desired.
\end{proof}

\noindent Theorem~\ref{shoenfieldabsoluteness} makes a short work out of many fairly involved manual checks. However, it has the disadvantage of needing the assumption that the model $M$ contains all ordinals, which is typically not necessary for the conclusion.

\begin{example}
Suppose that $M$ is a transitive model of set theory containing all ordinals, $M\models f\colon G\times X\to X$ is a continuous minimal flow of an interpretable group on a compact space. Let $\pi\colon G, X\to \hat G, \hat X$ be an interpretation. Then the flow $f$ is again interpreted
as a flow by analytic absoluteness~\ref{absolutenesstheorem}. In fact, the interpreted flow $\hat f$ will be minimal again: the set $C=\{K\in K(X)\colon K$ is $f$-invariant$\}$ is defined by a $\Pi_1$ formula, the formula will always define a closed subset of the hyperspace no matter which extension
of the model $M$ one considers, and therefore $\pi(C)=\{K\in K(\hat X)\colon
K$ is $\hat f$-invariant$\}$ by Shoenfield absoluteness~\ref{shoenfieldabsoluteness}. However, in $M$, the set $C$ contains just one element, namely $X$. Thus, $\pi(C)$ contains also only one element $\hat X$ and $\hat f$ is a minimal flow. It is possible to perform the whole computation by hand and thereby show that the conclusion holds even for transitive models $M$ which do not contain all ordinals.
\end{example}

\begin{example}
Suppose that $M$ is a transitive model of set theory containing all ordinals and $M\models K$ is a compact convex set. Suppose that $\pi\colon K\to\hat K$ is an interpretation; by analytic absoluteness, the convexity structure on $K$ gives rise to a convexity structure on $\hat K$. The space $C(K,\mathbb{R})$ is interpreted as $C(\hat K, \mathbb{R})$ by Theorem~\ref{cktheorem}. Now, the closed set of convex functions in $C(K, \mathbb{R})$ is defined by a $\Pi_1$ formula which absolutely defines a closed set. The conclusion of Shoenfield absoluteness~\ref{shoenfieldabsoluteness} is that its interpretation is the set of convex functions in $C(\hat K, \mathbb{R})$. A manual computation can remove the assumption that $M$ contains all ordinals.
\end{example}

\begin{example}
Suppose that $M$ is a transitive model of set theory, $M\models X$ is a locally convex topological vector space with interpretable topology, $K\subset X$ a compact convex set such that the set $L\subset K$ of all extreme points of $K$ is compact.
Suppose that $\pi\colon X\to\hat X$ is an interpretation.
Note that the set of extreme points of $K$ is defined by a $\Pi_1$ formula: $L=\{x\in K\colon\forall y_0, y_1\in K\ \forall r\in [0,1]\ x=ry_0+(1-r)y_1\to x=y_0\lor x=y_1\}$ in the structure including $X$, $[0,1]$, multiplication, addition, and the predicate for $K$.
Note also that it is not possible to apply Shoenfield absoluteness~\ref{shoenfieldabsoluteness} to argue that $\pi(L)$ is the set of all extreme points of $\pi(K)$ as it is not clear whether the set of all extreme points of $\pi(K)$ must be compact. Instead, it is necessary to resort to an interesting manual computation:

Work in the model $M$ for a moment. By the Krein--Milman theorem \cite[Theorem 3.23]{rudin:functional}, $K$ is the topological closure of the algebraic convex closure of $L$.
That is, $K={\mathrm{cl}}(\bigcup_n K_n)$ where for each number $n\in\gw$ write $I_n\subset [0,1]^n$ for the compact set of all $n$-tuples whose sum is $1$, and each $K_n$ is the image of $L^n\times I_n$ under the map $f(\vec x, \vec r)=\sum_i r_ix_i$. Note that each set $K_n\subset X$ is compact.

Step out of the model $M$. Each $\pi(K_n)$ is the image of $\pi(L^n)\times\pi (I_n)$ under the map $\pi(f)$. In other words,
$\bigcup_n\pi_n(K_n)$ is the algebraic convex closure of $\pi(L)$. But then, $\pi(K)={\mathrm{cl}}(\bigcup_n\pi(K_n))$ is the topological closure of the algebraic convex closure of $\pi(L)$. The set $\pi(L)\subset\hat X$ is compact by Corollary~\ref{compactcorollary}. By Milman's theorem \cite[Theorem 3.25]{rudin:functional}, all extreme points of $\pi(K)$ belong to the set $\pi(L)$. Also, the set $\pi(L)$ consists only of extreme points of $\pi(K)$ by the analytic absoluteness. In conclusion, $\pi(L)$ is exactly the set of all extremities of $\pi(K)$ as desired.
\end{example}

\begin{example}
The demand that $M$ contain all ordinals cannot be removed from the assumptions of Theorem~\ref{shoenfieldabsoluteness}. To show this, use the fact that the statement ``$x\in\baire$ is constructible'' is $\gS^1_2(x)$ \cite[Theorem 25.26]{jech:set} to find an effectively closed set $C\subset (\baire)^3$  such that for every $x\in\baire$, $x$ is constructible if and only if the projection of $C_x$ into the second coordinate is not all of $\baire$. Now, suppose that $V=L$ and $M$ is a countable transitive model which contains a point $x\in\baire$ such that $M\models x$ is not constructible. Then $M\models$the formula $\forall z\ \langle x, y, z\rangle\notin C$ with a free variable $y$ absolutely defines a closed subset of $\baire$, namely the empty set. However, it defines a nonempty set in $V$.
\end{example}

\section{Quotients}

In this section, I will show that certain commonly encountered types of quotient spaces are interpreted in the expected way.
This is connected to the evaluation of interpretations of certain types of surjective maps.
As the first case, recall that a function $f\colon X\to Y$ is \emph{open} if images of open sets are open.

\begin{theorem}
\label{openmappingtheorem}
\textnormal{(Open mapping theorem)}
Let $M$ be a transitive model of set theory and $M\models\langle X, \tau\rangle, \langle Y, \gs\rangle$
is an interpretable space and $f\colon X\to Y$ is an open continuous function. Let $\pi\colon X\to\hat X$
and $\chi\colon Y\to\hat Y$ be interpretations. Then 

\begin{enumerate}
\item $\hat f$ is an open continuous function from $\hat X$ to $\hat Y$;
\item whenever $O\in\tau$ and $P\in\gs$ are open sets such that $f''O=P$, then $\hat f''\pi(O)=\chi(P)$.
\end{enumerate}
\end{theorem}

\noindent In particular, a continuous open surjection is interpreted as continuous open surjection.

\begin{proof}
It is enough to verify that if the function $f$ is surjective then its interpretation $\hat f$ is surjective, as both (1) and (2) then follow by applying this result to the restricted functions $f\restriction O$. Suppose that $y\in\hat Y$ is a point and work to find $x\in\hat X$ such that $\hat f(x)=y$. Define a collection $F$ of closed subsets of $\hat X$ by $F=\{{\mathrm{cl}}(\hat f^{-1}\pi(P))\colon P\in\gs$ and $y\in\chi(P)\}$. It will be enough to show
that $\bigcap F\neq 0$, since every point in the intersection must be mapped to $y$ by Theorem~\ref{functiontheorem}. 

To this end, in the model $M$ find a complete sieve
$\langle S, O(s)\colon s\in S\rangle$ on $X$. By induction on $n\in\gw$ build an inclusion-descending sequence $\langle s_n\colon n\in\gw\rangle$ of elements of the tree $S$ such that $y\in\chi(f''O(t_n))$.
This is easy to do. Start with $s_0=0$ and once $s_n$ is found, let $D_n=\{O(t)\colon t$ is a one-step extension of $s_n\}$, note that $\bigcup D_n=O_n$ and use the fact that $\bigcup\pi''D_n=\chi(f''O(s_n))$ to find a one-step extension
$s_{n+1}$ of $s_n$ such that $y\in\chi(f''O(s_n))$. This concludes the induction step.

Now, let $E=F\cup\{{\mathrm{cl}}(O(s_n))\colon n\in\gw\}$. Observe that the collection $E$ has finite intersection property. To see this, suppose that $n\in\gw$ is a number and $P\in\gs$ is an open set such that $y\in \chi(P)$. The set $\chi(f''O(s_n))\cap\chi(P)$ is an open subset of the space $\hat Y$
containing $y$, therefore nonempty, and as $\chi$ is an interpretation, the open set $f''O(s_n)\cap P\subset Y$ must be nonempty. Thus, there must be a point $x\in O_n$ such that $f(x)\in P$, and then $\pi(x)\in \pi(O(s_n))\cap \hat f^{-1}\chi(P_n)$ and the latter set is nonempty as desired.

The sieve $\langle S, \pi(O(s))\colon s\in S\rangle$ on the space $\hat X$ is complete by Theorem~\ref{sievetheorem}, and so $\bigcap E\neq 0$, showing that $\bigcap F\neq 0$ as desired.
\end{proof}

\begin{corollary}
\label{openquotientcorollary}
Suppose that $M$ is a transitive model of set theory and $M\models\langle X, \tau\rangle$ is an interpretable space and $E$ is a closed equivalence relation on it such that saturations of open sets are open,
and such that the quotient $X/E$ is a regular Hausdorff space.
Let $\pi\colon X\to\hat X$ be an interpretation.

\begin{enumerate}
\item The map $[x]_E\mapsto [\pi(x)]_{\pi(E)}$ extends to an interpretation of $X/E$ in the space $\hat X/\pi(E)$;
\item The interpretation of the quotient map $f\colon X\to X/E$ is the quotient map $\hat f\colon \hat X\to\hat X/\pi(E)$.
\end{enumerate}
\end{corollary}

\begin{proof}
The quotient map $f$ is open in the model $M$ by the assumptions on the equivalence relation $E$. By Theorem~\ref{openmappingtheorem}, it is interpreted as an open map $\hat f\colon \hat X\to Y$ where $Y$ is the interpretation of the
space $X/E$. Every open map is a quotient map, and so it is enough to show that the equivalence relation $F$ on $\hat X$ given by $x_0\mathrel Fx_1$ is equal to $\pi(E)$.

Note that $F$ is closed and it contains $\pi''E$ which is dense in the closed equivalence relation $\pi(E)$; so it is enough to show that if $O, P\in\tau$ are open sets such that $(O\times P)\cap E=0$ then $(\pi(O)\times\pi(P))\cap F=0$.
To see this, note that $(O\times P)\cap E=0$ is equivalent to $f''O\cap f''P=0$, which implies $\hat f''\pi(O)\cap\hat f''\pi(P)=0$ by Corollary~\ref{disjointrangecorollary}, which by the definition of $F$ indeed means that $(\pi(O)\times\pi(P))\cap F=0$ as desired.
 \end{proof}

\begin{example}
The corollary shows that the coset spaces for closed subgroups are interpreted in the expected way.
Suppose that $M$ is a transitive model of set theory and $M\models \langle G, \tau, \cdot\rangle$ is an interpretable topological group and $H\subset G$ is a closed subgroup. Let $E$ be the closed equivalence relation on $X$ defined by $x_0\mathrel Ex_1$ if $x_0\cdot x_1^{-1}\in H$. This is a closed equivalence relation such that
saturations of open sets are open. Now let $\pi\colon G\to\hat G$ be an interpretation. It is not difficult to verify that $\pi(\cdot)$ is a group operation, $\pi(H)$ is a closed subgroup, $\pi(E)$ is an equivalence relation connecting $x_0, x_1$ just in case $x_0\cdot x_1^{-1}\in\pi(H)$.
Corollary~\ref{openquotientcorollary} then shows that the interpretable coset space $G/H$ is interpreted as $\hat G/\pi(H)$.
\end{example}

Now, recall that a function $f\colon X\to Y$ is \emph{perfect} if it is continuous, surjective, images of closed sets are closed and preimages of singletons are compact.

\begin{theorem}
\label{perfectmappingtheorem}
\textnormal{(Perfect mapping theorem)}
Suppose that $M$ is a transitive model of set theory and $M\models \langle X, \tau\rangle$ is an interpretable space, $\langle Y, \gs\rangle$ is a regular Hausdorff space and $f\colon X\to Y$
is a perfect mapping. Let $\pi\colon X\to \hat X$ and $\chi\colon Y\to\hat Y$ be interpretations, and $\hat f\colon\hat X\to\hat Y$ the interpreted map. Then 

\begin{enumerate}
\item $\hat f$ is a perfect mapping;
\item whenever $O\in\tau$, $P\in\gs$ are open sets such that $f''X\setminus O=Y\setminus P$ then $\hat f(\hat X\setminus\pi(O))=\hat Y\setminus\chi(P)$.
\end{enumerate}
\end{theorem}

\begin{proof}
I will start with a small claim.

\begin{claim}
In the model $M$: if $O=\bigcup_{i\in I}O_i$ is a union of open sets, then $Y\setminus f''(X\setminus O)=\bigcup\{D\in\gs\colon f^{-1}D$ is covered by finitely many sets $O_i\}$.
\end{claim}

\begin{proof}
The right-to-left inclusion follows from the definitions. For the left-to-right inclusion, suppose that $y\in Y\setminus f''(X\setminus O)$ is a point. The set $f^{-1}\{y\}$ is a compact subset of $O$ and so there is a finite set $J\subset I$ such that $f^{-1}\{y\}\subset\bigcup_{i\in J}O_i$. Let $D=Y\setminus f''(X\setminus\bigcup_{i\in J}O_i)$, note that $y\in D$ and $D$ belongs to the union on the right hand side.
\end{proof}

Suppose that $y\in\hat Y$ is an arbitrary point; I will argue that $\hat f^{-1}\{y\}\subset\hat X$ is nonempty and compact.

To this end, in the model $M$ let $\langle S, O(s)\colon s\in S\rangle$ be a complete, finitely additive strong sieve on the space $X$. By induction on $n\in\gw$ build a descending sequence $\langle s_n\colon n\in\gw\rangle$ of nodes in $S$ such that
$y\in\chi(Y\setminus f''(X\setminus O(s_n))$. To do this, start with $s_0=0$ and in the induction step use the claim, the finite additivity of the sieve and the fact that $\chi$ is an interpretation. After the induction has been performed, note that
the $\pi$-image of the sieve $\langle S, O(s)\colon s\in S\rangle$ is a complete sieve on the space $\hat X$ by Theorem~\ref{sievetheorem}. Thus, the set $K=\bigcap_n{\mathrm{cl}}(\pi(O(s_n))\subset\hat X$
is compact. No finite subcollection of $E=\{\pi(f^{-1}(P))\colon P\in\gs$ and $y\notin\chi(P)\}$ can cover the set $K$, and therefore $K\setminus\bigcup E$ is a nonempty compact subset of $K$. The definitions and Theorem~\ref{functiontheorem} imply that $\hat f^{-1}\{y\}=K\setminus\bigcup E$.

Thus, we conclude that $\hat f$ is a surjective function such that preimages of singletons are compact. (2) now immediately follows by the application of this fact to the perfect mapping $f\restriction (X\setminus O)\colon X\setminus O\to Y\setminus P$. To conclude the proof, it remains to show that $\hat f$ is a closed mapping. Suppose that $C\subset X$ is a closed set
and $y\in Y$ is a point not in $\hat f''C$; I must find an open set $P\in\gs$ such that $y\in\chi(P)$ and $\hat f''C\cap\chi(P)=0$.
To find the set $P$, use the compactness of $K=\hat f^{-1}\{y\}$ to find an open set $O\in\tau$ such that $K\subset\pi(O)$ and $\pi(O)\cap C=0$. Let $P\in\gs$ be the open set equal to $Y\setminus f''(X\setminus O)$. By (2), $\hat f''C\subset\hat f''(X\setminus O)=\hat Y\setminus \chi(P)$. Thus, the set $P\in\gs$ works as required.
\end{proof}

\begin{corollary}
\label{perfectquotientcorollary}
Suppose that $M$ is a transitive model of set theory and $M\models\langle X, \tau\rangle$ is an interpretable space and $E$ is a closed equivalence relation on it such that saturations of closed sets are closed, equivalence classes are compact, and
 such that the quotient $X/E$ is a regular Hausdorff space.
Let $\pi\colon X\to\hat X$ be an interpretation.

\begin{enumerate}
\item The map $[x]_E\mapsto [\pi(x)]_{\pi(E)}$ extends to an interpretation of $X/E$ in the space $\hat X/\pi(E)$;
\item The interpretation of the quotient map $f\colon X\to X/E$ is the quotient map $\hat f\colon \hat X\to\hat X/\pi(E)$.
\end{enumerate}
\end{corollary}

\begin{proof}
The quotient map $f$ is perfect in the model $M$ by the assumptions on the equivalence relation $E$. By Theorem~\ref{openmappingtheorem}, it is interpreted as a perfect map $\hat f\colon \hat X\to Y$ where $Y$ is the interpretation of the
space $X/E$. Every perfect map is a quotient map, and so it is enough to show that the equivalence relation $F$ on $\hat X$ given by $x_0\mathrel Fx_1$ is equal to $\pi(E)$. This follows letter by letter the second paragraph of the
proof of Corollary~\ref{openquotientcorollary}.
 \end{proof}

\begin{example}
The corollary shows that gluing in interpretable spaces is interpreted in the expected way.
Suppose that $M$ is a transitive model of set theory and $X$ is the closed unit square in it, and $E$ is the equivalence relation on $X$ connecting $\langle 0, x\rangle$ with $\langle 1, 1-x\rangle$ and leaves all other points equivalent only to themselves.
The quotient $X/E$ is the M\" obius strip. The equivalence relation $E$ satisfies the assumptions of Corollary~\ref{perfectquotientcorollary} and so the interpretation of $X/E$ is the M\" obius strip again.
\end{example}

\begin{example}
Suppose that $M$ is a transitive model of set theory and $M\models H$ is a Hilbert space with a norm $\phi$. Let $S\subset H$ be the unit sphere, and $E$ the equivalence of linear dependence on $S$. The relation $E$
satisfies the assumptions of Corollary~\ref{perfectquotientcorollary}. The quotient $S/E$ is the projective Hilbert space of $H$. Thus, the interpretation of the projective Hilbert space in $M$ is a projective Hilbert space.
\end{example}

\section{Hyperspaces}

Recall that if $X$ is a topological space then $K(X)$ denotes the space of its nonempty compact subsets, equipped with Vietoris topology, generated by sets $\{K\in K(X)\colon K\subseteq O\}$ and $\{K\in K(X)\colon K\cap O=0\}$ as $O$ varies over all open subsets of $X$.

\begin{theorem}
\label{hyperspacetheorem}
Suppose that $M$ is a model of set theory and $M\models\langle X, \tau\rangle$ is an interpretable space. Suppose that $\pi\colon X\to\hat X$ is an interpretation. Let $K(X)$ be the hyperspace of $X$ as evaluated in $M$, and let $K(\hat X)$ be the hyperspace of $\hat X$.

\begin{enumerate}
\item The map $\pi\colon K(X)\to K(\hat X)$ can be extended to an interpretation of the hyperspace $K(X)$;
\item Whenever $O\in\tau$ is an open set, $\pi(\{K\in K(X)\colon K\subset O\})=\{K\in K(\hat X)\colon K\subset\pi(O)\}$
and $\pi(\{K\in K(X)\colon K\cap O\neq 0\})=\{K\in K(\hat X)\colon K\cap\pi(O)\neq 0\}$.
\end{enumerate}
\end{theorem}

\noindent The interpretation $\pi$ maps compact subsets of $X$ in $M$ to compact subsets of $\hat X$ by Corollary~\ref{compactcorollary},
so the map $\pi\colon K(X)\to K(\hat X)$ is well-defined.

\begin{proof}
Whenever $A,B\subset\tau$ are finite sets, let $U(A,B)\subset K(X)$ be the set of all $K\in K(X)$ such that for every $O\in A$ $K\subset O$ holds, and
for every $O\in B$ $K\cap O\neq 0$ holds. A similar definition will be used to generate open sets of the space $K(\hat X)$.
Note that $\gs=\{U(A,B)\colon A, B\subset\tau$ finite$\}$ is a basis of the space $K(X)$ in the model $M$ and $\hat\gs=\{U(\pi''A,\pi''B)\colon A, B\subset\tau$ finite$\}$ is a basis of the space $K(\hat X)$. The following claim
captures the conversation between the bases $\gs$ and $\hat\gs$.

\begin{claim}
Let $A,B\subset\tau$ be finite sets.

\begin{enumerate}
\item for every $K\in K(X)$, $K\in U(A,B)\liff \pi(K)\in U(\pi''A, \pi''B)$;
\item if $A_i, B_i\subset\tau$ for $i\in I$ are finite sets and $I$ is finite and $U(A,B)\subset\bigcup_iU(A_i, B_i)$, then $U(\pi''A, \pi''B)\subset \bigcup_iU(\pi''A_i, \pi''B_i)$.
\end{enumerate}
\end{claim}

\begin{proof}
For (1), just unravel the definition of the set $U(A,B)$ and use the fact that $\pi\colon X\to\hat X$ is an interpretation.
For (2), suppose that the conclusion fails, as witnessed by some set $L\in K(\hat X)$, $L\in U(\pi''A, \pi''B)\setminus\bigcup_iU(\pi''A_i, \pi''B_i)$. Then, there is a partition $I=I_0\cup I_1$ and sets $O_i\in A_i$ for $i\in I_0$ and $O_j\in B_j$ for $j\in I_1$
such that if $i\in I_0$ then $L\not\subset\pi(O_i)$ and if $j\in I_1$ then $L\cap\pi(O_j)=0$. This means that for every $i\in I_0$, the set $(\bigcap \pi''A\setminus\bigcup_{j\in I_1}\pi(O_j))\setminus O_i$ is nonempty, and for every $P\in B$, the set $\pi(P)\cap\bigcap\pi''A\setminus\bigcup_{j\in I_1}\pi(O_j)$ is nonempty. As $\pi\colon X\to\hat X$ is an interpretation, this in turn means that the sets $(\bigcap A\setminus\bigcup_{j\in I_1}O_j)\setminus O_i$ is nonempty for every $i\in I_0$,
and for every $P\in B$, the set $P\cap\bigcap A\setminus\bigcup_{j\in I_1}O_j$ is nonempty. The finitely many points from these nonempty open sets can be collected to form a finite set $K\in U(A,B)\setminus\bigcup_iU(A_i, B_i)$ in the model $M$. (2) follows.
\end{proof}

Now, I will argue that Proposition~\ref{testproposition} can be applied with the basis $\gs$ to show that
the canonical extension $\bar \pi$ of the map $\pi\colon K\mapsto\pi(K)$ from $K(X)$ to $K(\hat X)$ of Definition~\ref{extensiondefinition} is in fact an interpretation.

To this end, first observe that $\bar\pi(U(A,B))=U(\pi''A, \pi''B)$. For the right-to-left inclusion, note that $U(\pi''A,\pi''B)\subset K(\hat X)$ is an open set disjoint from $\pi''(K(X)\setminus U(A, B))$ by (1) of the claim. For the left-to-right inclusion, note that if $U(\pi''C, \pi''D)\subset K(\hat X)$ is an open set disjoint from $\pi''(K(X)\setminus U(A, B))$,
then $U(C,D)\subset U(A,B)$ and by (2) of the claim, $U(\pi''C, \pi''D)\subset U(\pi''A, \pi''D)$.

Work in the model $M$ and find a complete finitely additive sieve $\langle S, O(s)\colon s\in S\rangle$ on the space $X$. As in the proof of Theorem~\ref{interpretabletheorem} observe that $\langle S, U(O(s), 0)\colon s\in S\rangle$ is a complete sieve on the space $K(X)$.
Step out of the model $M$, and use Theorem~\ref{sievetheorem} to argue that $\langle S, \pi(O(s))\colon s\in S\rangle$
is a complete sieve on the space $X$ and Theorem~\ref{interpretabletheorem} to argue that $\langle S, U(\pi(O(s)), 0)\colon s\in S\rangle$ is a complete sieve on the space $K(\hat X)$. Thus, the canonical extension $\hat\pi$ moves a complete sieve on $K(X)$ to a complete sieve on $K(\hat X)$. Proposition~\ref{testproposition} now completes the proof of the theorem.
\end{proof}

\section{The compact-open topology}

\noindent The most usual spaces of continuous functions are interpreted in the expected way:

\begin{theorem}
\label{cktheorem}
Let $M$ be a transitive model of set theory and $M\models\langle X_0, \tau_0\rangle$, $\langle X_1, \tau_1\rangle$
are a compact Hausdorff space and a completely metrizable metric space respectively. Let $C(X_0, X_1)$ the the space of continuous functions from $X_0$ to $X_1$ with the compact-open topology as evaluated in $M$. Let $\pi_0\colon X_0\to\hat X_0$
and $\pi_1\colon X_1\to\hat X_1$ be interpretations. Let $C(\hat X_0, \hat X_1)$ be the space of continuous functions from $\hat X_0$ to $\hat X_1$ with the compact-open topology.  Then 

\begin{enumerate}
\item the map $\pi\colon f\mapsto (\pi_0\times\pi_1)(f)$ extends to
an interpretation of $C(X_0, X_1)$ to $C(\hat X_0, \hat X_1)$;
\item the evaluation function $(x, f)\mapsto f(x)$ from $X_0\times C(X_0, X_1)$ to $X_1$ is interpreted as
the evaluation function from $\hat X_0\times C(\hat X_0, \hat X_1)$ to $\hat X_1$.
\end{enumerate}
\end{theorem}

\begin{proof}
Let $d\in M$ be any complete metric on the space $X_1$; this turns $C(X_0, X_1)$ into a complete metric space with the metric $e(f, g)=\sup\{d(f(x), g(x))\colon x\in X_0\}$. The metric $d$ is interpreted by $\pi_1$ as a complete metric on the space $\hat X_1$ by Theorem~\ref{metrizablecorollary}. Let $\hat e$ be the complete metric on $C(\hat X_0, \hat X_1)$ defined by $\hat e(f, g)=\sup\{\pi_1(d)(f(x), g(x))\colon x\in\hat X_0\}$.

The first observation is that the map $\pi$ is an isometric embedding from $\langle C(X_0, X_1), e\rangle$ to $\langle C(\hat X_0, \hat X_1), \hat e\rangle$. To see this, suppose that $f, g\in C(X_0, X_1)$ are functions. The set $\{d(f(x), g(x))\colon x\in X_0\}$
is dense in $\{\pi_1(d)(\pi(f), \pi(g))\colon x\in \hat X_0\}$, so these sets have the same supremum and $e(f,g)=\hat e(\pi(f), \pi(g))$ must hold.

The second point is that the range of $\pi$ is dense in the space $C(\hat X_0, \hat X_1)$. For this, suppose that $f\in C(\hat X_0, \hat X_1)$ is a function and $\eps>0$ be a positive rational number. 

\begin{claim}
There is a sequence $y=\langle O_0(n, i)\in\tau_0, O_1(n, i)\in\tau_1\colon i\in I_n, n\in\gw\rangle$ such that for each number $n\in\gw$,

\begin{enumerate}
\item $I_n$ is a finite set, $\bigcup_iO_0(n, i)=X_0$,  and all the vertical sections of the set $\bigcup_i{\mathrm{cl}}(O_0(n, i)) \times {\mathrm{cl}}(O_1(n, i))$
have $d$-diameter at most $2^{-n}\eps$;
\item $\bigcup_i{\mathrm{cl}}(O_0(n+1,i))\times{\mathrm{cl}}(O_1(n+1,i))\subset\bigcup_iO_0(n,i)\times O_1(n,i)$;
\item $f\subset\bigcup_i\pi_0(O_0(n,i))\times\pi(O_1(n,i))$.
\end{enumerate}

\end{claim}

\begin{proof}
This is a straightforward compactness argument with the space $\hat X_0$. Note that $f\subset \hat X_0\times\hat X_1$, as a continuous image of the compact space $\hat X_0$, is compact.
\end{proof}

By a wellfoundedness argument with the transitive model $M$, there is an infinite sequence $z=\langle O_0(n, i)\in\tau_0, O_1(n, i)\in\tau_1\colon i\in I_n, n\in\gw\rangle$  in the model $M$ satisfying the first two items from the claim, and such that $z$ starts with the same tuple $\langle O_0(0, i), O_1(0, i)\colon i\in I_0\rangle$ as the sequence $y$
obtained in $V$ by an application of the claim. Let $g\subset X_0\times X_1$ be defined as
$\bigcap_n \bigcup_{i\in I_n} \langle {\mathrm{cl}}(O_0(n, i)\times{\mathrm{cl}}(O_1(n, i))\rangle$. It is not difficult to verify that $g\in C(X_0, X_1)$ is a continuous function and $\hat e(\pi(g), \pi(f))<\eps$. This proves that the range of $\pi$ is dense in the metric space $\langle C(\hat X_0, \hat X_1), \hat e\rangle$.

Now, by Theorem~\ref{metrizablecorollary}, the map $\pi$ extends to an interpretation of $C(X_0, X_1)$ in $C(\hat X_0, \hat X_1)$
so that $\pi(e)=\hat e$. The proof of item (2) is left to the reader.
\end{proof}

The interpretation of spaces of continuous functions opens the door to the interpretation of regular Borel measures.

\begin{theorem}
Suppose that $M$ is a transitive model of set theory and $M\models \langle X,\tau, \mathcal{B}\rangle$ is a locally compact Hausdorff space with Borel structure and $\mu$ is a regular Borel measure on $X$. Suppose that $\pi\colon X\to\hat X$ is an interpretation. Then there is a unique regular Borel measure $\hat\mu$ on $\hat X$ such that for every set $B\in\mathcal{B}$, $\mu(B)=\hat\mu(\pi(B))$.
\end{theorem}

\begin{proof}
I will need an easy general claim.

\begin{claim}
\label{veryeasyclaim}
Whenever $K\subset O$ are a compact and open subset of $X$ respectively then there is an open set $P\in\tau$ such that
$K\subset\pi({\mathrm{cl}}(P))\subset O$.
\end{claim}

\begin{proof}
Since $\pi\colon X\to\hat X$ is an interpretation, for every point $x\in K$ there are sets $O_x, P_x\in\tau$ such that
$x\in O_x$ and $\pi(O_x)\subset O$ and ${\mathrm{cl}}(P_x)\subset O_x$. By a compactness argument, find an open set $L\subset K$ such that $K\subset\bigcup_{x\in L}\pi(P_x)$ and let $P=\bigcup_{x\in L}P_x$. The set $P\in\tau$ works.
\end{proof}

First, handle the case of a compact space $X$. In this case, the measure $\mu$ must be finite by regularity.
In the model $M$, let $F\colon C(X, \mathbb{R})\to\mathbb{R}$ be the continuous linear operator of Lebesgue integration: $F(f)=\int f\ d\mu$.
By Theorem~\ref{cktheorem}, the space $C(X, \mathbb{R})$ is interpreted as $C(\hat X, \mathbb{R})$ via a a map which I will call $\pi$ again. By analytic absoluteness~\ref{absolutenesstheorem}, $\pi(F)$ is a continuous linear operator on $C(\hat X, \mathbb{R})$. By the Riesz representation theorem, there is a regular Borel measure $\hat\mu$ on $\hat X$ such that $\pi(F)$ is the integration with respect to $\hat \mu$. I claim that $\hat\mu$ works.

First, prove that for every open set $O\in\tau$, $\mu(O)=\hat\mu(\pi(O))$. For the $\leq$ inequivalence, if $q<\mu(O)$ is a rational number, then by the regularity of $\mu$ there is a compact set $K\in M$ such that $K\subset O$ and $\mu(K)>q$. By the Hausdorffness of the space $X$ in the model $M$, there is a continuous function $f\in C(X, [0,1])$ such that $f\restriction K=1$ and $f\restriction (X\setminus O)=0$. Then by analytic absoluteness $\supp(\pi(f))\subset\pi(O)$ and so $\hat \mu(\pi(O))>\pi(F)(\pi(f)=F(f)>q$ as required.
For the $\geq$ inequivalence, if $q<\hat\mu(\pi(O))$ is a rational number, then by the regularity of $\hat\mu$ there is a compact set
$K\subset \pi(O)$ such that $K\subset\pi(O)$ and $\hat\mu(K)>q$. By Claim~\ref{veryeasyclaim}, there is an open set $P\in\tau$ such that ${\mathrm{cl}}(P)\subset O$ and $K\subset\pi(P)$.
There is a continuous function $f\in C(X, [0,1])$ such that $f\restriction{\mathrm{cl}}(P)=1$ and $f\restriction (X\setminus O)=0$.
Then $\mu(O)\geq F(f)=\pi(F)(\pi(f))\geq\hat\mu({\mathrm{cl}}(\pi(P)))\geq\hat\mu(K)>q$ as desired.

Now, it follows that for every compact set $K\subset X$ in the model $M$, $\mu(K)=\hat\mu(\pi(K))$ since $\mu(K)=\mu(X)-\mu(X\setminus K)=\hat \mu(\hat X)-\hat\mu(\pi(X\setminus K))=\hat\mu(\pi(K))$ by the work on open sets.
By regularity of the measure $\mu$, it also follows that for every set $B\in\mathcal{B}$, $\mu(B)=\hat\mu(\pi(B))$ as desired.

For the uniqueness of the regular Borel measure $\hat\mu$,
by the inner regularity of $\hat\mu$ and Claim~\ref{veryeasyclaim}, for every open set $O\subset X$, $\hat\mu(O)=\sup\{\hat\mu(\pi(P)\colon P\in\tau, \pi(P)\subset O\}$ and so the values of $\hat\mu$ on open sets are uniquely determined by the demand on agreement with $\mu$. The values of $\hat\mu$ on other Borel sets are uniquely determined by the outer regularity demand on $\hat\mu$.

For the general case of a locally compact space $X$ and a (possibly infinite) regular Borel measure $\mu$ on $X$,
for every open set $O\in\tau$ such that ${\mathrm{cl}}(O)$ is compact write $\mu_O$ for the restriction of the measure $\mu$ to ${\mathrm{cl}}(O)$,
$\hat\mu_O$ for the unique measure on the compact set ${\mathrm{cl}}(\pi(O))\subset \hat X$ obtained from $\mu_O$ by the work on the compact case, and let $\hat\mu$ be the measure on $\hat X$ defined by $\hat\mu(B)=\sup\{\hat\mu_O(B\cap{\mathrm{cl}}(\pi(O)))\colon O\in\tau$ and ${\mathrm{cl}}(O)\subset X$ is compact$\}$. The verification of the required properties of the measure $\hat\mu$ is routine and left to the reader.
\end{proof}

\begin{example}
Suppose that $M$ is a transitive model of set theory and $M\models G$ is a topological group with locally compact Hausdorff topology. Let $\mu$ be the left invariant Haar measure on $G$ in the model $M$. Let $\pi\colon G\to\hat G$ be an interpretation. Then $\hat\mu$ is a left invariant Haar measure on the group $\hat G$. For the sake of brevity, I deal with the case of a compact group $G$; the slightly more involved general case
is left to the interested reader. Note that the $\mu$-integration linear operator $F$ on $C(G,\mathbb{R})$ in the model $M$ is invariant under left shifts by elements of $G$. By analytic absoluteness~\ref{absolutenesstheorem}, this is also true of the interpretation $\hat F$ of $F$,
and so the measure $\hat\mu$ obtained from $\hat F$ must be left-invariant as well.
\end{example}

\section{Banach spaces}
\label{functionalsection}

In this section, I provide several theorems on the interpretation of the usual operations and concepts surrounding Banach spaces.

\begin{theorem}
\label{easybanachtheorem}
Topological vector space over $\mathbb{R}$ (if interpretable) is interpreted as a topological vector space over $\mathbb{R}$.
A normed Banach space is interpreted as a normed Banach space. The following properties of Banach spaces are preserved by the interpretation functor:

\begin{enumerate}
\item local convexity;
\item uniform convexity.
\end{enumerate}
\end{theorem}

\begin{proof}
These are all elementary consequences of analytic absoluteness~\ref{absolutenesstheorem}. Suppose that $M$ is a transitive model of set theory and $M\models X$ is a topological vector space. First of all, the axioms of topological vector space are all $\Pi_1$, and therefore they survive the interpretation process. Now suppose that $\phi$ is a complete norm
on $X$; then $d(x, y)=\phi(x-y)$ is a complete metric on $X$. Let $\pi\colon X\to\hat X$ be an interpretation. The interpretation $\pi(\phi)$ satisies the triangle inequality by analytic absoluteness, since the triangle inequality is a $\Pi_1$ statement. By Theorem~\ref{metrizablecorollary}, $\hat X$ is just the completion of $\pi''X$ under the metric $\pi(d)$. In particular, $\pi(d)$ is a complete metric on $\hat X$, by analytic absoluteness $(\pi (d))(x, y)=\pi(\phi)(x-y)$ and so $\pi(\phi)$ is a complete norm on $\hat X$. Therefore, a Banach space is interpreted as a Banach space, and its complete norms are interpreted as complete norms.

The local convexity is the statement that every (basic) open neighborhood in $X$ is a union of convex open sets. Now, if an
open set $O\subset X$ in $\tau$ is convex, then $\pi(O)$ is convex again, since convexity is a $\Pi_1$ property of the set.
Also, interpretations commute with arbitrary unions and so if $O=\bigcup_iO_i$ is a union of convex open sets in the model $M$, then $\pi(O)=\bigcup_i\pi(O_i)$ is a union of open convex sets as desired.

The uniform convexity is a $\Pi_1$ statement and therefore survives the interpretation process again.
\end{proof}

Now, I am ready to show that various operations on Banach spaces commute with the interpretation functor. 
The most popular operation on Banach spaces is certainly taking a dual, either with the weak${}^*$ topology, or with the topology obtained from the dual norm.  For a normed Banach space $X$ with norm $\phi$, write $X^*$ for its normed dual with the dual norm $\phi^*$. The bracket $\langle\rangle_X$ denotes the evaluation map from $X\times X^*$
to $\mathbb{R}$: $\langle x, x^*\rangle=x^*(x)$.

\begin{theorem}
\label{hardbanachtheorem}
Suppose that $M$ is a transitive model of set theory and $X\models M$ is a Banach space with norm $\phi$, with dual $X^*$ and dual norm $\phi^*$.
Let $\pi\colon X\to\hat X$ be an interpretation. $\pi(\phi)$ is a complete norm on $\hat X$; write $(\hat X)^*$
for the dual of $\hat X$ with the dual norm $\pi(\phi)^*$. Then 

\begin{enumerate}
\item $X^*$ is interpreted as a closed subset of $(\hat X)^*$;
\item the norm $\phi^*$  is interpreted as the restriction of the norm $\pi(\phi)^*$;
\item the bracket $\langle\rangle_X$ is interpreted as the restriction of the bracket $\langle\rangle_{\hat X}$;
\item if $M\models X$ is uniformly convex then the normed dual of $X$ is interpreted as all of $\hat X^*$.
\end{enumerate}
\end{theorem}

\begin{proof}
Let $x^*$ be a continuous linear functional from $X$ to $\mathbb{R}$ in the model $M$. By Theorem~\ref{functiontheorem}, $\pi(x^*)$ is a continuous function from $\hat X$ to $\mathbb{R}$; by analytic absoluteness~\ref{absolutenesstheorem}, the function $\pi(x^*)$ is linear and so an element of $(\hat X)^*$.

Let $B\subset X$ be the $\phi$-unit ball; so $\pi(B)\subset\hat X$ is the $\pi(\phi)$-unit ball. Since $\pi''(x^*)\subset\pi(x^*)$ and $\pi''B\subset\pi(B)$ are dense sets, it is the case that $(x^*)''(B)$ is a dense subset
of $(\pi(x^*))''\pi(B)$, and so $\phi^*(x^*)=(\pi(\phi)^*)(\pi(x^*))$.
It follows that the map $x^*\mapsto\pi(x^*)$ is a norm and metric preserving embedding from $\langle X^*, \phi^*\rangle$ to $\langle (\hat X)^*, \pi(\phi)^*\rangle$ which also commutes with the brackets.

Let $\chi\colon X^*\to Y$ be an interpretation of the dual space $X^*$. By Theorem~\ref{metrizablecorollary}, the space $Y$ is just the completion of $\chi''X^*$ under the norm $\chi(\phi^*)$. As a result, there is a unique isometric embedding $h\colon \langle Y, \chi(\phi^*)\rangle \to \langle(\hat X)^*, \pi(\phi)^*\rangle$ such that $h(\chi(x^*))=\pi(x^*)$. Clearly,
$(h\circ\chi)''(\langle\rangle_X)$ is a subset of the bracket $\langle\rangle_{\hat X}$ and $(h\circ \chi)''\phi^*$ is a subset of the norm $(\pi(\phi))^*$ on $(\hat X)^*$. Since in both cases the functions in question are continuous, it follows that $(h\circ\chi)(\langle\rangle_X)$ is equal to the restriction of the bracket $\langle\rangle_{\hat X}$ to $\hat X\times \rng(h)$ and the $(h\circ\chi)(\phi^*)$ is just the restriction of $(\pi(\phi))^*$ to $\rng(h)$. This completes the verification of the first three items.

For the last item, I will use the following general claim of independent interest:

\begin{claim}
\label{weakstarclaim}
The set $\{\pi(x^*)\colon x^*\in X^*\}\subset (\hat X)^*$ is dense in the weak${}^*$ topology.
\end{claim}

\begin{proof}
Recall that the weak${}^*$ topology is just the topology of pointwise convergence of linear functionals in $(\hat X)^*$.
Thus, it will be enough, given a linear functional $x^*\in (\hat X)^*$, points $x_i\in\hat X$ and open sets $O_i\subset\mathbb{R}$ for $i\in n$, such that for all $i\in n$ $x^*(x_i)\in O_i$ holds, to produce a linear functional $x^{**}\in X^*$ such that for all
$i\in n$, $\pi(x^{**})(x_i)\in O_i$ holds.

Suppose for simpicity that the points $x_i\in \hat X$ are linearly independent. Find a real number $\eps>0$ such that
for all $i\in\gw$, the interval $(x^*(x_i)-\eps, x^*(x_i)+\eps)$ is a subset of $O_i$. Find a natural number $r$ such that $\pi(\phi)(x^*)\leq r$. Find points $y_i\in X$ such that

\begin{itemize}
\item $\{y_i\colon i\in n\}$ is a linearly independent collection;
\item $\phi(y_i)\geq\phi^*(x_i)$;
\item $\phi^*(x_i-\pi(y_i))<\eps/2r$.
\end{itemize}

This is easy using the fact that $\pi''X\subset \hat X$ is dense. Find rational numbers $t_i$ for $i\in\gw$ such that 
$\phi(x_i)-\eps/2<t_i<\phi(x_i)$ for each $i\in n$. In the model $M$, use the Hahn--Banach theorem to find a linear functional
$\phi(x^{**})\leq r$ and $x^{**}(y_i)=t_i$ for each $i\in n$. Then, for each $i\in n$, it is clear that
$\pi(x^{**})(x_i)=\pi(x^{**})(y_i)+\pi(x^{**})(x_i-y_i)\in (t-\eps/2, t+\eps/2)\subset O_i$, which completes the proof of the claim.
\end{proof}

For the conclusion of last item, it is enough to show that the interpretation of $X^*$ is dense in $(\hat X)^*$ in the sense of the dual norm $\pi(\phi)^*$: since these are both
complete metric spaces, the density necessitates the conclusion that the interpretation of $X^*$ is in fact equal to $(\hat X)^*$.
Since the interpretation of $X^*$ is a convex subset of $(\hat X)^*$, it is enough to show that it is dense in the sense of the weak topology of $(\hat X)^*$ by \cite[Theorem 3.12]{rudin:functional}. Now (the key step in the proof), since $M\models X$ is uniformly convex, then so is $\hat X$ by Theorem~\ref{easybanachtheorem},
so $\hat X$ is reflexive by Milman--Pettis theorem \cite{pettis:convex} and the weak topology on $(\hat X)^*$ coincides with the weak${}^*$ topology. Thus,
the desired conclusion follows from Claim~\ref{weakstarclaim}.
\end{proof}

\begin{example}
In a transitive model $M$ of set theory, consider the spaces $X=\ell_1$ and $Y=\ell_\infty$ with their usual norms. The space $X$, being separable, is interpreted as $\ell_1$. On the other hand, whenever $a\subset\gw$ is a set which is not in $M$, the characteristic function of $a$
is an element of $\ell_\infty$ which is not in the closure of $Y$. Thus, the interpretation of $Y$ will be a proper closed subset of the normed dual of the interpretation of $X$.
\end{example}

\begin{theorem}
\label{weakstartheorem}
Suppose that $M$ is a transitive model of set theory and $M\models \langle X, \tau\rangle$ is a Banach space and $Y$ the unit ball of its dual with the weak${}^*$ topology.
Suppose that $\pi\colon X\to\hat X$ is an interpretation and write $\hat Y$ for the unit ball of the dual of $\hat X$ with the 
weak${}^*$ topology.

\begin{enumerate}
\item The map $\pi\colon Y\to\hat Y$ extends to an interpretation;
\item the bracket on $X\times Y$ is interpreted as the bracket on $X\times Y$.
\end{enumerate}
\end{theorem}

\begin{proof}
The weak${}^*$ topology is just the pointwise convergence topology. Thus, define basic functions on $X$ as the functions $g$ whose domain is a finite subset of $X$ and whose values are real intervals with rational endpoints. For each basic function $g$ define a basic open set $U(g)\subset Y$ as the set of all points $y\in Y$ such that for every $x\in\dom(g)$, $y(x)\in g(x)$. Let $\gs=\{U(g)\subset Y\colon g$ is a basic function on $X\}$; this is a basis of the topology on the space $Y$. Similarly, define basic functions $g$ on $\hat X$ and their corresponding open sets $U(g)$ on the space $\hat Y$. Let $\hat\gs=\{U(\pi''g)\subset\hat Y\colon g$ is a basic function on $X\}$. 

\begin{claim}
$\hat\gs$ is a basis for the topology on $\hat Y$.
\end{claim}

\begin{proof}
Suppose that $h$ is a basic function on $\hat X$ and $y\in\hat Y$ belongs to $U(h)$. Write $h=\{\langle x_i, q_i\rangle i\in I\}$ and let $\eps>0$ be a rational number such that for every $i\in I$, $(y(x_i)-\eps, y(x_i)+\eps)\subset q_i$. Use the density of $\pi''X$ in $\hat X$ to find points $x'_i\in X_i$ and intervals $q'_i$ with rational endpoints so that the norm of $x_i-\pi(x'_i)$ is less than $<\eps/8$ and $(y(x_i)-\eps/4, y(x_i)+\eps/4)\subset q'_i\subset (y(x_i)-\eps/2, y(x_i)+\eps/2)$.
Let $g=\{\langle x'_i, q'_i\rangle\colon i\in I\}$ and use the linearity and bounded norm of operators in $Y$ to conclude that
$x\in U(\pi''g)\subset U(h)$.
\end{proof}

\begin{claim}
Suppose that $g$ is a basic function on $X$.

\begin{enumerate}
\item whenever $y\in Y$ then $y\in U(g)$ if and only if $\pi(y)\in U(\pi''g)$;
\item if $h_i$ for $i\in I$ and some finite set $I$ are basic functions on $X$ and $U(g)\subset\bigcup_iU(h_i)$, then $U(\pi''g)\subset\bigcup_iU(\pi''h)$.
\end{enumerate}
\end{claim}

\begin{proof}
For the first item, just unravel the definition of the set $U(g)$ and of the map $\pi$ on operators. For the second item, suppose that the conclusion fails and $y\in\hat Y$ is an operator in $U(\pi''g)\setminus\bigcup_iU(\pi''h_i)$. Let $X_0\subset X$ be the finite-dimensional space generated by $\dom(g)\cup\bigcup_i\dom(h_i)$. By analytic absoluteness~\ref{absolutenesstheorem}, there must be an operator $z_0\in M$ on $X_0$ of norm $\leq 1$ which is in $U(g)\setminus\bigcup_iU(h_i)$.  By the Hahn-Banach theorem in the model $M$, $z_0$ can be extended to an operator $z$ on all of $X$ of norm $\leq 1$. Then $z\in U(g)\setminus\bigcup_iU(h_i)$ as desired.
\end{proof}

Now, let $\bar\pi$ be the canonical extension of the map $\pi\colon Y\to\hat Y$ as described in Definition~\ref{extensiondefinition}. First argue that $\bar\pi(U(g))=U(\pi'' g)$. For the right-to-left inclusion note that $U(\pi''g)\subset\hat Y$ is an open set disjoint from $\pi''(Y\setminus U(g))$ by (1) of the claim. For the left-to-right inclusion, if $h$ is a basic function on $X$ and $U(\pi''h)\subset\hat Y$
is disjoint from $\pi''(Y\setminus U(g))$ then $U(h)\subset U(g)$ by (1) of the claim, and so $U(\pi''h)\subset U(\pi''g)$ by (2) of the claim. Now, the canonical extension $\bar\pi$ must be an interpretation of $Y$ in $\hat Y$ by Proposition~\ref{testproposition}, since the space $\hat Y$ is compact.
\end{proof}

\section{Faithfulness}
\label{faithfulnesssection}

It is interesting to see how the interpretations behave when there are more models around. The behavior in the category of interpretable spaces is the expected one. Outside of the category, there is an instructive counterexample.

\begin{theorem}
\label{faithfulnesstheorem}
Suppose that $M_0\subset M_1$ are transitive models of set theory, $M_0\models \langle X_0, {\tau}_0\rangle$ is an interpretable space, $M_1\models \pi_0\colon\langle X_0, {\tau}_0\rangle\to \langle X_1, {\tau}_1\rangle$ is an interpretation over $M_0$, and $\pi_1\colon\langle X_1, {\tau}_1\rangle\to\langle X_2, {\tau}_2\rangle$ is an interpretation over $M_1$. Then $\pi_1\circ\pi_0\colon\langle X_0, {\tau}_0\rangle\to\langle X_2, {\tau}_2\rangle$ is an interpretation over $M_0$.
\end{theorem}

\begin{proof}
The composition map is clearly a preinterpretation, and this would hold of any topological space $X_0$. If $\langle S, O(s)\colon s\in S\rangle$ is a complete sieve on $X_0$ in the model $M_0$, then by a repeated use of Theorem~\ref{sievetheorem}
the pair $\langle S, \pi_0(O(s))\colon s\in S\rangle$ is a complete sieve on $X_1$ in the model $M_1$, the pair $\langle S, \pi_1(\pi_0(O(s)))\colon s\in S\rangle$ is a complete sieve on $X_2$, and therefore the composition $\pi_1\circ\pi_0$
is an interpretation as desired.
\end{proof}

\begin{example}
\label{productexample}
The conclusion of Theorem~\ref{faithfulnesstheorem} fails for $X=\gw^{\gw_1}$.
To see, this I will first prove a characterization theorem for interpretations of the space $X$ in generic extensions in the presence of the continuum hypothesis.

\begin{claim}
\textnormal{(with Justin Moore)}
Assume the Continuum Hypothesis. Let $P$ be a partial ordering. The following are equivalent:

\begin{enumerate}
\item $P\Vdash$ the interpretation of $X^V$ is $X^{V[G]}$ with the identity map;
\item $P$ adds no reals and preserves the statement ``$T$ has no uncountable branches'' for every tree $T$ of size $\aleph_1$.
\end{enumerate}
\end{claim}

\noindent Here, by $X^{V[G]}$ I mean the product $\gw^{\gw_1^V}$ as evaluated in $V[G]$.

\begin{proof}
It is not difficult to show that to prove (1), one has to prove the following. Let $F$ be a collection of finite partial functions from $\gw_1$ to $\gw$
such that $\forall x\in X\ \exists f\in F\ f\subset x$; then
$P\Vdash\forall x\in X^{V[G]}\ \exists f\in F\ f\subset x$. To this end, look at the tree $T$ of all functions whose domain is some countable ordinal, range is a subset of natural numbers, and $t\in T$ implies that for no $f\in F$,
$f\subset t$. The assumptions imply that $T$ has no cofinal branches. Suppose now that (2) holds. Let $\dot x$ be a $P$-name for a new element of $X$. Since $P$ adds no new reals, all initial segments of $\dot x$ are forced to be in the ground model.
And since $P\Vdash\dot x$ is not a branch through $\check T$, there must be an initial segment of $\dot x$ which is not in $T$ and therefore contains some $f\in F$ as a subset. Thus, (1) holds.

Now, suppose that (2) fails. There are two distinct cases. Suppose first that $P$ adds a new real. Use the CH assumption to find an enumeration $\langle z_\ga\colon\gw\leq\ga<\gw_1\rangle$ of all elements of $\baire$.
For every $\gw\leq\ga<\gw_1$, consider the open set $O_\ga=\{x\in X\colon x\restriction x(\ga)=z_\ga\restriction x(\ga)\}$. Note that $\bigcup_\ga O_\ga=X$ since for every element $x\in X$ there is an ordinal $\ga$ such that $x\restriction\gw=z_\ga$
and then $x\in O_\ga$. However, if $\dot x$ is a $P$-name for any function from $\gw_1$ to $\gw$ such that $\dot x\restriction\gw\notin V$ and for every $\gw\leq\ga <\gw_1$ $\dot x(\ga)$ is a number so large that $z_\ga$ and $\dot x\restriction\gw$
differ below $\dot x(\ga)$, then $P\Vdash\dot x$ does not belong to the natural interpretation of the set $O_\ga$ for any ordinal $\ga$. In other words, an open cover of $X^V$ does not cover $X^{V[G]}$ and (1) fails.

Suppose now that $P$ adds no new reals and instead adds an uncountable branch through some tree $T$ of size $\aleph_1$ which has no uncountable branches in the ground model. One can assume that every countable increasing sequence in $T$ has a unique supremum, every terminal node of $T$ is on a limit level,
and nonterminal nodes split into exactly two immediate successors. For every $t\in T$ let $X_t$ be the set $\{s\in T\colon s\leq t\}$ equipped with discrete topology. The product $Y=\prod_tX_t$ is naturally homeomorphic to a closed subspace of $X$, so it will
be enough to show that there is an open cover of $Y^V$ which is not a cover of $Y^{V[G]}$. Consider the following open subsets of $Y$: $O_t=\{y\in Y\colon y(t)=t\}$ for terminal nodes $t\in T$, $Q_t=\{y\in Y$: if $y(t)=t$ then for neither of the two immediate successors $s$ of $t$, $y(s)=s\}$ for nonterminal nodes $t\in T$, and $P_{st}=\{y\in Y\colon y(s)$ is not compatible with $y(t)\}$ and $R_{st}=\{y\in Y\colon y(t)<y(s)\leq t\}$ for any two nodes $s,t\in T$.

On one hand, in the $P$ extension, the product $Y^{V[G]}$ is not covered by the union of (the natural interpretations of) these open sets: if $b\subset T$ is a cofinal branch in $V[G]$ then the point $y\in Y^{V[G]}$ defined by
$y(t)=$largest element of $b$ which is $\leq t$ does not belong to any of the open sets.  On the other hand, in the ground model the product $Y$ is covered by the union of these open sets: if $y\in Y$ fell out of all of them then define
$b=\rng(y)$ and observe that

\begin{itemize}
\item $b$ is a linearly ordered subset of $T$ by the definition of $P_{st}$;
\item $b$ is countable as $T$ has no cofinal branches in the ground model;
\item $b$ does not have a maximum. Such a maximum $t$ would have to have $y(t)=t$
by the definition of the sets $R(s,t)$, but then if $t$ is nonterminal then $y\in Q_t$ and if $t$ is terminal then $y\in O_t$;
\item $b$ does not have a limit ordertype, since for the supremum $t$, $y(t)<t$ would have to hold and
$y$ would belong to one of the sets $R_{st}$.
\end{itemize}

\noindent Since the last two items cover all possibilities, a contradiction is reached showing that $Y^V$ is covered by the open sets indicated and (1) again fails.
\end{proof}

It is now a routine matter to start with a model of CH and construct generic extensions $V\subset V[G]\subset V[H]$ such that both $V[G]$ and $V[H]$
are $\gs$-closed extensions of $V$ and $V[G]$ contains a branchless tree $T$ of size $\aleph_1$ which does have branches in $V[H]$. For example, to obtain $V[G]$ just add an $\gw_1$-tree $T$ with countable approximations (it will be a Suslin tree by \cite[Theorem 15.23]{jech:set}) and then add a generic branch through it to form $V[H]$. It is well known that $V[H]$ is a $\gs$-closed extension of $V$.
Once this is done, the claims show that the interpretation of $X^V$ in $V[H]$ over $V$ is $X^{V[H]}$, the interpretation in $V[G]$ is $X^{V[G]}$, and the interpretation of $X^{V[G]}$ is not $X^{V[H]}$, in violation of the conclusion of Theorem~\ref{faithfulnesstheorem}.
\end{example}

\begin{theorem}
\label{submodeltheorem}
Suppose that $\hat X$ is an interpretable space and $M$ is an elementary submodel of a large structure containing $\hat X$ as an element
and some basis of $\hat X$ as an element and a subset. Let $\pi\colon \bar M\to M$ be the inverse of the transitive collapse of $M$ and $X=\pi^{-1}(\hat X)$.
Then $\pi\colon X\to\hat X$ is an interpretation of the space $X$ over the model $\bar M$.
\end{theorem}

\begin{proof}
The map $\pi$ commutes with finite intersections and arbitrary unions in the model $\bar M$ simply because $\pi$ is an elementary embedding. Moreover, the range of $\pi$ generates the topology of the space $\hat X$ since it contains a basis by the assumptions.
Thus, the map $\pi$ is a preinterpretation, and this would be true for any topological space $X$. To prove that $\pi$ is an interpretation, find a basis $\gs$ of $\hat X$ which is an element and a subset of the model $M$.
Given a complete sieve $\langle S, O(s)\colon s\in S\rangle$ for the space $X$, refining and thinning out if necessary it is possible to amend it so that it uses only sets from the basis $\gs$ and so that if $t\in S$ and $s_0, s_1$ are immediate successors
of $t$ then $O(s_0)\neq O(s_1)$. Such a sieve has size $\leq|\gs|$ and one such a sieve must belong to the model $M$ by elementarity; it is then even a subset of $M$.
Then, the interpretation of the complete sieve $\pi^{-1}(\langle S, O(s)\colon s\in S\rangle)$ on $X$ is the complete sieve $\langle S, O(s)\colon s\in S\rangle$ on $\hat X$ and therefore $\pi\colon X\to\hat X$ is an interpretation by Theorem~\ref{sievetheorem}.
\end{proof}

\begin{example}
Whenever $X$ is a second countable space which is not Polish and $M$ is a countable elementary submodel of a large structure containing $X$, then the conclusion of the theorem fails since interpretations of spaces over countable $M$ must be Polish.
\end{example}

\section{Preservation theorems}
\label{preservationsection}

In this section I will show that certain properties of topological spaces survive the interpretation process. This is to say, if $M$ is a transitive model of set theory and $M\models X$ is a topological space with property $\phi$ and $\pi\colon X\to\hat X$ is a topological interpretation, then $\hat X$ has property $\phi$. There are very many open questions.

In the category of compact Hausdorff spaces, many topological properties are preserved simply because every open cover of the interpretation has a finite refinement whose elements are in the range of the interpretation map.
This immediately gives the following:

\begin{proposition}
Suppose that $M$ is a model of set theory and $M\models X$ is a compact Hausdorff space and $f\colon X\to X$ is a continuous map. Suppose that $\pi\colon X\to\hat X$ and $\hat f\colon \hat X\to\hat X$ are interpretations.

\begin{enumerate}
\item if $M\models X$ is connected then $\hat X$ is connected;
\item if $M\models X$ is totally disconnected then $\hat X$ is totally disconnected;
\item the Lebesgue covering dimension of $X$ as computed in $M$ is equal to that of $\hat X$;
\item the topological entropy of $f$ as computed in $M$ is equal to that of $\hat f$.
\end{enumerate}
\end{proposition}

\noindent I do not know if such categories as Lyusternik--Schnirelman category, small inductive dimension or large inductive dimension are necessarily preserved by interpretations of compact Hausdorff spaces.
In the broader category of interpretable spaces, much more complicated behavior is possible. I include just one preservation schema which yields many results.

\begin{definition}
\label{gooddefinition}
A property $\phi$ of open covers of topological spaces is \emph{good} if

\begin{enumerate}
\item it is upwards absolute: suppose that $M$ is a transitive model of set theory and $M\models X$ is an interpretable space and $C$ is an open cover of $X$ with $\phi(C)$. If $\pi\colon X\to\hat X$ is an interpretation, then $\phi(\pi''C)$ holds;
\item it is diagonalizable: if $X$ is an interpretable space, $C$ is an open cover with $\phi(C)$ and for every $O\in C$, $D_O$ is an open cover of some open set containing ${\mathrm{cl}}(O)$ as a subset such that $\phi(D_O)$ holds,
then there is a refinement of $\bigcup_OD_O$ which has property $\phi$;
\item if a finite cover fails to have $\phi$ then so do all of its refinements.
\end{enumerate}

\noindent Say that $X$ is $\phi$-compact if every open cover has a refinement which satisfies $\phi$, and $X$ is locally $\phi$-compact if for every point $x\in X$ and every neighborhood $x\in O$ there is an open set $x\in P\subset O$ such that
$P$ is $\phi$-compact.
\end{definition}

\noindent Several traditional properties of topological spaces can be expressed as local $\phi$-compactness for a good $\phi$:

\begin{example}
Local metacompactness. Consider the property $\phi$ of open covers $C$ saying ``every point is contained in only finitely many elements of $C$'', or in standard terminology ``$C$ is pointwise finite''. Then $\phi$ is good: $\phi(C)$ says that the intersection $\bigcap_nO_n$ is empty, where $O_n$ is the open set of points contained in more than $n$ many elements of the cover $C$. This statements is absolute by Theorem~\ref{boreltheorem}. For the diagonalization, if $C$ is a pointwise finite cover of $X$ and
for each $O\in C$ there is an pointwise finite open cover $D_O$ of some open set containing ${\mathrm{cl}}(O)$, then the collection $\{P\cap O\colon P\in D_O\}$ is a pointwise finite refinement of $\bigcup_OD_O$.
\end{example}

\begin{example}
Local paracompactness.  Consider the property $\phi$ of open covers $C$ saying ``every point has a neighborhood which has nonempty intersection with only finitely many elements of $C$'', or in standard terminology ``$C$ is locally finite''. Then $\phi$ is good: $\phi(C)$ says that the set $\{O\subset X\colon O$ is an open set with nonempty intersection with only finitely many elements of $C\}$ is a cover of $X$, and this is preserved by interpretations. For the diagonalization, if $C$ is a locally finite cover of $X$ and
for each $O\in C$ there is a locally finite open cover $D_O$ of some open set containing ${\mathrm{cl}}(O)$, then the collection $\{P\cap O\colon P\in D_O\}$ is a locally finite refinement of $\bigcup_OD_O$.
\end{example}

\begin{example}
Local Lindel{\" o}fness. Consider the property $\phi$ of open covers $C$ saying that $\bigcup C$ is covered by countably many elements of $C$.
\end{example}

\begin{example}
Local connectedness. Consider the property $\phi$ of open covers saying $C$ that for every two sets $P, Q\in C$ there are sets $\{O_i\colon i\in n\}$ in $C$ such that $O_0=P, O_{n-1}=Q$, and $O_i\cap O_{i+1}\neq 0$ for every $i\in n-1$.
\end{example}

\noindent Many properties such as (local) complete normality are not expressible in this way.

\begin{theorem}
Suppose that $\phi$ is a good property of covers. Suppose that $M$ is a transitive model of set theory $M\models\langle X, \tau\rangle$ is a locally $\phi$-compact interpretable space. Let $\pi\colon X\to\hat X$ be an interpretation. Then $\hat X$ is locally $\phi$-compact;
in fact, if $M\models O\in\tau$ is $\phi$-compact then $\pi(O)$ is $\phi$-compact.
\end{theorem}

\begin{proof}
Suppose for contradiction that $P\in\tau$ is a $\phi$-compact set, and $C$ is an open cover on $\pi(P)$ which shows that $\pi(P)$ fails to be $\phi$-compact--thus, $P$ has no refinement with property $\phi$. Refining if necessary, I may assume that $C\subset\pi''\tau$.
Let $\langle S, O(s)\colon s\in S\rangle$ be a complete sieve on the space $X$ in the model $M$. By induction on $n\in\gw$ build nodes $s_n\in S$ and sets $P_n\in\tau$ so that

\begin{itemize}
\item $s_0=0$ and $s_{n+1}$ is an immediate successor of $s_n$;
\item $P=P_0\supset P_1\supset\dots$ are $\phi$-compact sets in the model $M$;
\item $C\restriction\pi(P_n)$, the cover of $\pi(P_n)$ consisting of intersections of elements of $C$ with $\pi(P_n)$, has no refinement with property $\phi$;
\item ${\mathrm{cl}}(P_n)\subset P_n\cap O(s_n)$.
\end{itemize}

\noindent This is not difficult to do. At the induction step, work in the model $M$ and write $P_n=\bigcup D$ where $Q\in D$ implies that $Q$ is a $\phi$-compact open set and $\bar Q\subset P_n\cap O(t)$ for some immediate successor $t$ of $s_n$.
This is possible as the space $X$ is locally $\phi$-compact. I claim that there must be $O\in D$ such that $C\restriction\pi(O)$ has no refinement with property $\phi$ which makes the induction step immediately possible.
Suppose for contradiction that this fails. Use the $\phi$-compactness of $P_n$ in the model $M$ to find a refinement $D'$ of $D$ with the property $\phi$ such that the closure of every element of $D'$ is a subset of some
element of $D$. Step out of the model $M$. By the upward absoluteness, $\pi''D'$ is a cover of $\pi(P_n)$ with the property $\phi$. For every $O\in D$ there is a cover $D_O$ of $\pi(O)$ which is a refinement of $C\restriction \pi(O)$ and has the property $\phi$
by the contradictory assumption. By the diagonalizability, there is is a refinement of $\bigcup_OD_O$ which is a cover of $\pi(P_n)$ with the property $\phi$. However, this contradicts the third item of the induction hypothesis.

Once the induction process is complete, let $F$ be the collection $\{{\mathrm{cl}}(\pi(P_n))\colon n\in\gw\}\cup\{\hat X\setminus O\colon O\in C\}$ of closed subsets of the space $\hat X$. The collection $F$ has the
finite intersection property since no set $P_n$ can be covered by finitely many sets $O$ such that $\pi(O)$ is an element of $C$: such a finite cover $E$ would have to be in the model $M$,
$\pi''E$ would be a refinement of $C\restriction P_n$, so it would fail $\phi$, thus $M\models\lnot\phi(E)$ by the absoluteness clause of Definition~\ref{gooddefinition}, and $M\models E$ witnesses the failure of $\phi$-compactness
of the set $P_n$ by the third clause of Definition~\ref{gooddefinition}. This would violate the second item of the induction hypothesis at $n$. Now, since the $\pi$-image of the sieve is complete by Theorem~\ref{sievetheorem},
the collection $F$ has a nonempty intersection, containing some element $x\in\hat X$. Then $x\notin\bigcup C$, contradicting the assumption that $C$ was a cover.
\end{proof}

\begin{corollary}
\label{preservationcorollary}
The following properties of interpretable spaces are preserved by interpretations over transitive models of set theory:

\begin{enumerate}
\item local paracompactness;
\item local metacompactness;
\item local Lindel{\" o}fness;
\item local connectedness.
\end{enumerate}
\end{corollary}

The corollary can be sharpened in certain more special circumstances. For example, by a theorem of Frol{\' \i}k \cite{frolik:paracompact}, a regular Hausdorff space $X$ is paracompact {\v C}ech complete if and only if there is a perfect map $f\colon X\to Y$ onto a completely metrizable space 
$Y$. Thus, if $M\models X$ is {\v C}ech complete and paracompact, then the interpretation of $X$ is again {\v C}ech complete and paracompact since the perfect map is interpreted as a perfect map again by Theorem~\ref{perfectmappingtheorem}
and the completely metrizable space is interpreted as completely metrizable again.

\section{Interpretable Borel spaces}
\label{borelsection}

There are important spaces which are not interpretable, including $\mathbb{Q}$ or $C_p(\mathbb{R})$. In general, the interpretation of such spaces exhibits various pathologies. In a good class of examples though, one can adjust the notion of interpretation so that the resulting functor commutes with many natural operations on such spaces.

\begin{definition}
A \emph{Borel space} is a triple $\langle X, \tau, \mathcal{B}\rangle$ where $X$ is a set, $\tau$ is a topology, and $\mathcal{B}$ is the $\gs$-algebra of Borel sets.
\end{definition}

Borel spaces have a natural notion of interpretation between transitive models of set theory:

\begin{definition}
\label{i3definition}
Suppose that $M$ is a transitive model of set theory and $M\models\langle X, \tau, \mathcal{B}\rangle$ is a Borel space. A \emph{preinterpretation} of the Borel space is a map $\pi\colon X\to\hat X$, $\pi\colon\tau\to\hat\tau$ and $\pi\colon \mathcal{B}\to\hat{\mathcal{B}}$ where
$\langle \hat X,\hat\tau, \hat{\mathcal{B}}\rangle$ is a Borel space and

\begin{enumerate}
\item for every $x\in X$ and every $O\in\tau$, $x\in O\liff \pi(x)\in\pi(O)$;
\item $\pi\restriction\tau$ commutes with finite intersections and arbitrary unions in the model $M$, $\pi(0)=0$ and $\pi(X)=\hat X$;
\item $\pi''\tau$ generates the topology $\hat\tau$;
\item $\pi\restriction\mathcal{B}$ extends $\pi\restriction\tau$ and it commutes with complements, countable unions and intersections in the model $M$.
\end{enumerate}

\noindent An \emph{interpretation} is the largest preinterpretation in the sense of reducibility if it exists.
\end{definition}

Thus, Theorem~\ref{boreltheorem} says that the notion of interpretation of Borel spaces essentially coincides with interpretations of topological spaces in the category of interpretable topological spaces. The development of the theory of interpretations of Borel spaces closely follows the topological case. 

\begin{theorem}
If $M$ is a transitive model of set theory and $M\models\langle X, \tau, \mathcal{B}\rangle$ is a Borel space whose topology is regular Hausdorff, then its preinterpretation exists, it is unique up to equivalence of preinterpretations, and its topology is regular Hausdorff again.
\end{theorem}

\begin{theorem}
\label{botheorem}
The interpretation functor on interpretable Borel spaces commutes with the operation of taking a Borel subset.
\end{theorem}

\begin{definition}
An \emph{interpretable Borel space} is one isomorphic to a Borel subset of an interpretable topological space, with the inherited Borel structure.
\end{definition}

\begin{theorem}
The interpretation functor on interpretable Borel spaces commutes with the operation of countable product.
\end{theorem}

As a final note, I will show that in a very common case, that of proper bounding forcing extensions, the notions of interpretation of topological and Borel topological spaces essentially coincide.

\begin{theorem}
\label{boundingtheorem}
Suppose that $M$ is a transitive model of set theory such that

\begin{enumerate}
\item every countable subset of $M$ is covered by a set in $M$ which is countable in $M$;
\item every function in $\baire$ is pointwise dominated by some function in $\baire\cap M$.
\end{enumerate}

Suppose $M\models\langle X, \tau, \mathcal{B}\rangle$ is a regular Hausdorff space with a Borel $\gs$-algebra. If $\pi\colon \langle X, \tau\rangle\to\langle \hat X, \hat\tau\rangle$ is a interpretation of the topological space then $\pi$ extends to an interpretation of the Borel space.
\end{theorem}

\begin{proof}
It will be enough to show that if $\pi$ is a topological preinterpretation the it can be extended to a Borel-topological interpretation, since then the direct limits used to compute the topological and Borel-topological interpretations coincide. The argument follows closely the proof of Theorem~\ref{boreltheorem}. Work in the model $M$. By induction on the complexity of a Borel code $c\in M$, define an ordering $T_c$. An element $p\in T_c$ will be a tuple whose first element will be a nonempty closed subset of $X$ denoted as $E(p)$.

\begin{itemize}
\item if $c=\{A\}$ for a nonempty closed or open set $A\subset X$, then $T_c$ is an ordering of all pairs $p=\langle E, n\rangle$ such that $E\subset A$ is closed and nonempty and $n\in\gw$. The ordering is defined by $\langle E, n\rangle>\langle F, m\rangle$ if $F\subseteq E$ and $n<m$;
\item if $c=\{\bigcup, D\}$ for some countable set $D$ of codes, $T_c$ is the ordering of all triples $\langle E, d, u\rangle$
where $E\subset X$ is a nonempty closed set, $d\in D$, $u\in T_d$, and $E\subseteq E(u)$. I set $\langle E, d, u\rangle>\langle F, e, v\rangle$ if $F\subseteq E$, $d=e$ and $v<u$;
\item if $c=\{\bigcap, D\}$ for some countable set $D=\{d_i\colon i\in\gw\}$ of codes with a fixed enumeration, let
$T_c$ be the set of all tuples $\langle E, u_i\colon i\in n\rangle$ where $E\subset X$ is a nonempty closed set, $u_i\in T_{d_i}$
for all $i\in n$, and $E\subseteq\bigcap{i\in n}E(u_i)$. The ordering is defined by $\langle E, u_i\colon i\in n\rangle>\langle F, v_i\colon i\in m\rangle$ if
$F\subseteq E$, $m>n$, and for all $i\in n$ $u_i>v_i$ in $T_{d_i}$.
\end{itemize}

\begin{claim}
\label{boundingclaim1}
In the model $M$, if $p$ is an infinite descending sequence in $T_c$ then $\bigcap_nE(p(n))\subset B_c$.
\end{claim}

\begin{claim}
\label{boundingclaim2}
In $V$, if $x\in \hat B_c$ is a point then there is an infinite descending sequence $p$ in $T_c$ such that $x\in\bigcap_n\pi(E(p(n))$.
\end{claim}

Now, as in the proof of Theorem~\ref{boreltheorem}, it is just enough to show that if $c\in M$ is a Borel code and $\hat B_c\neq 0$ then $B_c\neq 0$. Suppose then that $\hat B_c\neq 0$. Let $T$ be the tree of all finite attempts tobuild an infinite descending sequence in $T_c$; so $T\in M$. Use Claim~\ref{boundingclaim2} to find an infinite branch $p$ in $T$ such that $\bigcap_n\pi(E(p(n)))\neq 0$. Use the assumptions (1) and (2) of the theorem to find a finitely branching tree $S\subset T$ in the model $M$ such that $b\subset S$. For every number $n\in\gw$ let $C_n=\bigcup\{t(n)\colon t\in S\}$. The set $C_n\subset X$ is a finite union of closed sets and as such it is closed. The set $\bigcap_n\pi(C_n)\subset\hat X$ is nonempty, containing the point $x$,
and since the map $\pi$ is a topological interpretation, the set $\bigcap_nC_n\subset X$ must be nonempty. Let $y\in X$
be a point in this intersection. A compactness argument with the finitely branching tree $S$ shows that there must be an infinite branch $c\subset S$ such that $y\in\bigcap_nc(n)$. Claim~\ref{boundingclaim1} then shows that $y\in B_c$ as desired.
\end{proof}

\noindent Neither of the assumptions of Theorem~\ref{boundingtheorem} can be removed as the following two examples show.

\begin{example}
Let $\kappa$ be an uncountable regular cardinal and let $\langle X, \tau, \mathcal{B}\rangle$ be the subspace of $2^\kappa$ consisting of characteristic functions of finite sets with a Borel structure.

\begin{enumerate}
\item The interpretation of the Borel space $X$ in every extension is $X$ itself with the identity function;
\item In every extension of $V$ in which cofinality of $\kappa$ is countable, the interpretation of the topological space $X$ icontains a characteristic function of an infinite set.
\end{enumerate}
\end{example}

\noindent Thus, for example if the Continuum Hypothesis holds and the Namba forcing cofinalizes $\kappa=\gw_2^V$ to $\gw$
the extension contains no new reals but still in the topological interpretation of the space $X$ it is impossible to faithfully define interpretations of Borel sets.

\begin{proof}
Observe that the space $X$ is an $F_\gs$ subset of the space $2^\kappa$. (1) then follows immediately from  Theorems~\ref{compactproducttheorem} and~\ref{subspacetheorem}. (2) is more difficult. Pass to an extension in which there is a set $a\subset\kappa$ of ordertype $\gw$ which is cofinal in $\kappa$. Let $g\in 2^\kappa$ be the characteristic function of the set $a$. It will be enough to show that the identity map $\pi$
from $X$ to $X\cup\{g\}$ extends to a topological interpretation of $X$. To see why the map $\pi$ commutes with unions of open sets, suppose that in $V$, $O=\bigcup_{i\in I}O_i$ is a union of open subsets of the space $X$. Suppose for simplicity that $O=X$ and the set $O_i$ are obtained as intersections of basic open subsets of $2^\kappa$ with $X$;
I must prove that there is $i\in\gw$ such that $g\in O_i$. To this end, work in $V$ and define the function $h\colon\kappa\to\kappa$
by $h(\gb)=$the least ordinal $\ga$ such that for every finite set $b\subset\ga$, the characteristic function of $b$ belongs to some open set $O_i$ such that the support of $O_i$ is a subset of $\ga$. Since the cardinal $\kappa$ is regular, the function $g$ is well defined. Comparing ordertypes, it is easy to conclude that there must be an ordinal $\ga\in\kappa$ such that the ordinal interval $[\ga, f(\ga)]$ contains no elements of the set $a$. Then, there is an index $i\in I$ such that the characteristic function of $a\cap\ga$ is in the open set $O_i$, and the support of $O_i$ is a subset of $f(\ga)$. Since the map $g$ is coincides with the characteristic function of $a\cap\ga$ below $f(\ga)$, it follows that $g\in O_i$ as desired.
\end{proof}

\begin{example}
\label{cpexample}
Let $\langle X, \tau, \mathcal{B}\rangle$ be the Borel space of continuous functions from $[0,1]$ to $[0,1]$ with pointwise convergence.

\begin{enumerate}
\item The interpretation of the Borel space $X$ in every extension is the space of continuous functions from
$[0,1]$ to $[0,1]$ with the topology of pointwise convergence on $[0,1]\cap V$;
\item In every extension in which there is an unbounded real, the interpretation of the topological space $X$ contains a discontinuous function and it does not extend to a Borel-topological interpretation.
\end{enumerate}
\end{example}

\begin{proof}
Note that the space $X$ is a Borel subset of $Y=[0,1]^{[0,1]}$ with the product topology, consisting of uniformly continuous functions. To see how it is expressed as a Borel set, for positive rationals $\eps, \gd>0$ let $C_{\eps\gd}=\{f\in Y\colon\forall x_0, x_1\in [0,1]\ |x_0-x_1|\leq\eps\to |y_0-y_1|\leq\gd\}$ and observe that $C_{\eps\gd}\subset Y$ is a closed set. Then
$X=\bigcap_\eps\bigcup_\gd C_{\eps\gd}$. 

For (1), apply Theorems~\ref{compactproducttheorem} and~\ref{subspacetheorem} to see that the interpretation of the Borel space $X$ in any extension is just the space $\hat X$ of all uniformly continuous functions from $[0,1]\cap V$ to $[0,1]$. Every $f\in\hat X$ has a unique extension to a continuous function on $[0,1]$, and the extension map will be a homeomorphism between $\hat X$ and the space of all continuous functions from $[0,1]$ to $[0,1]$ with pointwise convergence on the set $[0,1]\cap V$.

For (2), for every $n\in\gw$ let $f_n\colon [0,1]\to\mathbb{R}$ be a continuous function such that $\supp(f_n)\subset (2^{-n-1}, 2^{-n})$
and $f_n(2^{-n}-2^{-n-2})=1$. In the extension, let $a\subset\gw$ be a set such that for every increasing function $h\in\baire\cap V$ there is $m\in\gw$ such that the interval $[m, h(m)]$ contains no numbers in $a$; such a set $a$ exists as the extension is assumed to contain an unbounded real.  Let $g\colon [0,1]\cap V\to\mathbb{R}$
be the function defined as $g=\sum_{n\in a}f_n$. I will show that the identity map $\pi\colon X\to X\cup\{g\}$ extends to a topological preinterpretation of $X$.

Indeed, suppose that $O=\bigcup_iO_i$ is a union of basic open sets in $V$. Suppose for simplicity that $O=X$ and the sets $O_i$ are basic open. It will be enough to show that there is an index $i\in I$ such that $g\in O_i$. Work in $V$. For every finite set $b\subset\gw$ let $k_b=\sum_{n\in b}f_n$. This is a function in $X$, therefore it belongs to one of the sets $O_i$. Let
$h\in\baire$ be a function such that for every $m\in\gw$ and every $b\subset m$ there is $i\in I$ such that $k_b\in O_i$ and
$\supp(O_i)\cap (0, 2^{-h(m)})=0$. Now, the choice of the set $a$ implies that there is $m\in\gw$ such that there are no elements of $a$ between $m$ and $h(m)$. Let $b=a\cap m$, and let $i\in I$ be an index such that $k_b\in O_i$ and $\supp(O_i)\cap (0, 2^{-h(m)})=0$. Since outside of the interval $(0, 2^{-h(m)})$ the functions $g$ and $k_b$ are equal, it follows that $g\in O_i$ as desired.

Now, the map $\pi$ does not extend to a Borel-topological preinterpretation of $X$, since the $F_{\gs\gd}$ set
$\{f\in X\colon f(0)=0\land\forall m\exists n>m\ f(2^{-n}-2^{-n-2})=1\}$ is empty in $V$ while the only candidate for its interpretation in the space $X\cup \{g\}$ is the nonempty set $\{g\}$. It follows that the topological interpretation of $X$ (which must contain a copy of $X\cup \{g\}$) cannot be extended to a Borel-topological interpretation of $X$.
\end{proof}

\section{Comparison with Fremlin's work}
\label{fremlinsection}

In \cite{fremlin:top}, Fremlin defines an interpretation functor for topological spaces in the special case in which $V$ is a generic extension of $M$. Here, I will show that the Fremlin interpretation and the interpretations introduced in the current paper coincide on the categories of interpretable topological spaces and interpretable Borel spaces.

First, it is necessary to identify the definition of the Fremlin interpretation. In \cite[2Aa]{fremlin:top}, for a topological space $\langle X, \tau\rangle$ in $M$ a space $\langle \tilde X, \tilde\tau\rangle$ is defined together with a map $A\mapsto\tilde A$ for all subsets $A\subset X$ in the model $M$. I will call the map $\pi\colon X\to\tilde X$ given by $\{\pi(x)\}= \tilde{\{x\}}$ together with the map $\pi\colon\tau\to\tilde\tau$ given by $\pi(O)=\tilde O$ the \emph{Fremlin interpretation} of the topological space $X$. Given a function $\phi\colon X\to Y$ between topological spaces in the model $M$, a map $\tilde\phi\colon\tilde X\to\tilde Y$ is defined in \cite[2Ca]{fremlin:top}. I will call $\tilde\phi$ the \emph{Fremlin interpretation} of the map $\phi$.

Note that Fremlin interpretations are defined as specific sets. This is in contradistinction with the treatment of the present paper, where interpretations are defined up to a certain commutative diagram. However, up to this insignificant notational detail, the Fremlin interpretations and the interpretations of the present paper coincide on the categories of interpretable topological spaces and interpretable Borel spaces. This is the contents of the following theorem.

\begin{theorem}
\label{fremlintheorem}
Let $V$ be a generic extension of $M$. 

\begin{enumerate}
\item The Fremlin interpretation of an interpretable topological space is an interpretation in the sense of Definition~\ref{i1definition};
\item the Fremlin interpretation of a continuous function between interpretable topological spaces is an interpretation in the sense of Definition~\ref{i2definition};
\item the Fremlin interpretation of an interpretable Borel space is an interpretation in the sense of Definition~\ref{i3definition}.
\end{enumerate}
\end{theorem}

\begin{proof}
To shorten the expressions in this proof, the word ``interpretation'' without ``Fremlin'' in front of it will always mean an interpretation in the sense of Definitions~\ref{i1definition}, ~\ref{i2definition}, and ~\ref{i3definition} of the current paper. I will show in turn that a Fremlin interpretation is a preinterpretation, and that it is an interpretation in the case of a compact Hausdorff space, a $G_\gd$ subset of a compact Hausdorff space,
and a continuous open image of a $G_\gd$ subset of a compact Hausdorff space. This will prove (1).

To see that the Fremlin interpretation $\pi\colon \langle X, \tau\rangle\to\langle\tilde X, \tilde\tau\rangle$ is a preinterpretation, consult \cite[2Ab(iv)]{fremlin:top} to see that the the map preserves finite intersections and consult \cite[2Ab(vi)]{fremlin:top}
to see the preservation of arbitrary unions of open sets. Finally, the topology  $\tilde\tau$ is generated by $\pi''\tau$ essentially by its definition; cf.\ \cite[2Ac]{fremlin:top}.

Now, the target space of a Fremlin interpretation $\pi$ of a compact Hausdorff space $X\in M$ is again compact Hausdorff by \cite[4A]{fremlin:top}. Corollary~\ref{compact1corollary} shows that any preinterpretation of $X$ with a compact Hausdorff target space is an interpretation and so $\pi$ is an interpretation.  If $Y\subset X$ is a universally Baire (in particular, $G_\gd$) subset of a topological space then the Fremlin interpretation of $Y$ is (up to an obvious commutative diagram) equal to the restriction of the Fremlin interpretation of $X$ to $Y$ by \cite[2Ce, 2Cg ff.]{fremlin:top}. By Corollary~\ref{subspacecorollary}, the Fremlin interpretation of a \v Cech complete space (as a $G_\gd$ subset of a compact Hausdorff space) is an interpretation.

Now, suppose that $\langle X, \tau\rangle$ is a \v Cech-complete space, $\langle Y, \gs\rangle$ is a regular Hausdorff space, and $f\colon X\to Y$ is a continuous open surjective map, all in the model $M$. Let $\pi\colon \langle X, \tau\rangle\to \langle \tilde X, \tilde\tau\rangle$, $\chi\colon \langle Y, \gs\rangle\to\langle\tilde Y, \tilde\gs\rangle$, and $\tilde f\colon\tilde X\to\tilde Y$ be Fremlin interpretations of these objects. Let $\nu\colon \langle Y, \gs\rangle\to \langle\hat Y, \hat\gs\rangle$ be an interpretation.
I must prove that $\nu$ is reducible to $\chi$.

Since $\chi$ is a preinterpretation of $Y$ by the second paragraph of the present proof, there is a map $h\colon\tilde Y\to\hat Y$ which reduces the preinterpretation $\chi$ to the interpretation $\nu$. It follows from \cite[2Cc]{fremlin:top} and the uniqueness part of Theorem~\ref{functiontheorem} that $h\circ\tilde f$ is an interpretation of the function $f$. By Theorem~\ref{openmappingtheorem}, the map $h\circ\tilde f$ is onto $\hat Y$, and this can happen only if $h$ is a bijection. Then $h^{-1}$ reduces $\nu$ to $\chi$ as desired.

(2) now follows immediately from \cite[2Cc]{fremlin:top} and the uniqueness part of Theorem~\ref{functiontheorem}. (3) follows from (1), \cite[2Ce, 2Cg ff.]{fremlin:top} and Theorem~\ref{botheorem}; both Fremlin interpretations and interpretations of Borel spaces introduced in the current paper commute with taking a Borel subset of an interpretable topological space.
\end{proof}

\bibliographystyle{plain}
\bibliography{odkazy}

\end{document}